\documentclass[reqno]{amsart}
\usepackage{amsmath}
\usepackage{amssymb}
\usepackage{xstring}
\usepackage{latexsym}
\usepackage{color}
\calclayout

\newcommand{\half}[1][1]{\frac{#1}{2}}
\newcommand{\op}[1]{\operatorname #1}
\newcommand{\ip}[2]{\left\langle{#1},{#2}\right\rangle}
\newcommand{\I}[1]{\int_\Omega #1 \,dx dy}
\newcommand{\Ip}[2]{\int_\Omega #1\cdot#2 \,dx dy}
\newcommand{\pt}{\partial_t}
\newcommand{\pn}[1]{\partial_{#1}}
\newcommand{\px}{\partial_x}
\newcommand{\py}{\partial_y}
\newcommand{\pxi}{\partial_i}

\newcommand{\dt}{\dfrac{d}{dt}}
\newcommand{\sR}{\mathbb{R}}
\newcommand{\sL}{L^2}
\newcommand{\sH}{H^1}
\newcommand{\sSob}[1]{H^{#1}}
\newcommand{\n}[1]{\Vert#1\Vert}
\renewcommand{\a}[1]{\left\vert#1\right\vert}

\newcommand{\nl}[1]{\n{#1}_{L^2}}
\newcommand{\nL}[1]{\n{#1}_{L^4}}
\newcommand{\nh}[1]{\n{\og#1}_{\sL}}
\newcommand{\nH}[1]{\n{\ol#1}_{\sL}}

\newcommand{\ob}[2]{\left(#1\cdot\og\right)#2}

\newcommand{\og}{\nabla}
\newcommand{\ol}{\Delta}
\newcommand{\odiv}{\og\cdot}
\newcommand{\oI}{\op{I}_h}
\newcommand{\ib}[3]{\ip{\ob{#1}{#2}}{#3}}
\newcommand{\Ib}[3]{\Ip{\ob{#1}{#2}}{#3}}
\newcommand{\Ibb}[4]{\I{#1_{#2}\pn{\subxy{#2}}#3#4}}
\newcommand{\subxy}[1]{%
	\IfEqCase{#1}{%
		{1}{x}
		{2}{y}
	}[\PackageError{tree}{Undefined option to tree: #1}{}]%
}%
\newcommand{\PT}{{\scriptstyle\mathcal{P}}}
\newcommand{\ut}{\tilde{u}}
\newcommand{\vt}{\tilde{v}}
\newcommand{\wt}{\tilde{w}}
\newcommand{\cI}{c_1}
\newcommand{\cIa}{c_2}
\newcommand{\cIb}{c_3}
\newcommand{\cL}{c_L}
\newcommand{\cB}{c_B}
\newcommand{\cM}{c_M}
\newcommand{\C}{C}
\newcommand{\cT}{c_T}
\newcommand{\nua}{(\alpha-\beta)}
\newcommand{\nub}{\alpha - \beta}
\newcommand{\cupa}{(1+T\pi^2\nua)\nua}
\newcommand{\cupb}[1][10]{#1\pi^2\nua^2}
\newcommand{\cupT}{\cM \nua^2}
\newcommand{\ar}[1]{\frac{\nh{{#1}_1}}{2\pi\nl{{#1}_1}}}
\newcommand{\ars}[1]{\frac{\nh{{#1}_1}^2}{4\pi^2\nl{{#1}_1}^2}}

\newcommand{\arsT}[1]{\frac{\nH{{#1}_1}^2}{4\pi^2\nh{{#1}_1}^2}}
\newcommand{\ag}[2][1+]{\left(#1\ln#2\right)}
\newcommand{\avw}{Z}
\newcommand{\avwf}{\left(\nh{v}^2+\nh{w}^2\right)}

\newcommand{\lnTwo}[1]{\left(1+\ln\frac{\nH{#1_1}}
	{2\pi\nh{#1_1}}\right)^{1/2}}
\newcommand{\lntwo}[1]{\left(1+\ln\frac{\nH{#1_1}}
	{2\pi\nh{#1_1}}\right)}
\DeclareMathOperator*{\esssup}{ess\,sup}
\newtheorem{theorem}{Theorem}[section]

\newtheorem{proposition}[theorem]{Proposition}

\newtheorem{system}[theorem]{System}
\newtheorem{lemma}[theorem]{Lemma}
\theoremstyle{remark}
\newtheorem{remark}[theorem]{Remark}

\title[Abridged data assimilation for the 2D MHD]{Continuous data assimilation for the magnetohydrodynamic equations in 2D using one component of the velocity and magnetic fields}
	\author{Animikh Biswas$^{1}$}
\address{${}^{1,2}$ Department of Mathematics and Statistics\\
University of Maryland-Baltimore County\\ Baltimore, MD 21250, USA.}
\email[A. Biswas]{abiswas@umbc.edu}
\author{Joshua Hudson$^{2, \dagger}$}
\email[J. Hudson]{joshuahudson@umbc.edu}
\address{${}^{3,4}$ Department of Mathematics\\ University of Nebraska-Lincoln\\ Lincoln, NE 68588-0130, USA.}
\author{Adam Larios$^3$}
\email[A. Larios]{alarios@unl.edu}
\author{Yuan Pei$^4$}
\email[Y. Pei]{ypei4@unl.edu}

\thanks{${}^\dagger$ Corresponding author. {\it Email:} joshuahudson@umbc.edu}
\thanks{The research of Animikh Biswas and Joshua Hudson was partially supported by NSF grant DMS 1517027.}

\begin{document}
\begin{abstract}
We propose several continuous data assimilation (downscaling) algorithms based on feedback control for the 2D magnetohydrodynamic (MHD) equations.  We show that for sufficiently large choices of the control parameter and resolution and assuming that the observed data is error-free,  the solution of the controlled system converges exponentially (in $L^2$ and $H^1$ norms) to the reference solution independently of the initial data chosen for the controlled system.  Furthermore, we show that a similar result holds when controls are placed only on the horizontal (or vertical) variables, or on a {\it single} Els\"asser variable, under more restrictive conditions on the control parameter and resolution.  Finally, using the data assimilation system, we show the existence of {\it abridged} determining modes, nodes and volume elements.
\end{abstract}
\maketitle

\section{Introduction}
In the study of solar storms, space weather forecasting, earth's geodynamo, and other areas, predicting the motion of fluids with magnetic properties is a central concern.  The governing equations are often taken to be the magnetohydrodynamic (MHD) equations, or some modification of them.  These equations are notoriously difficult to solve both analytically and computationally.  Moreover, accurately initializing the system is challenging due to the sparsity of the available data.  Fortunately, data is often given not just at a single time, but can be streaming in (e.g., from devices monitoring space plasma dynamics), or given in history (e.g., from surface geomagnetic observations, which in the earth can be traced back up to 7000 years \cite{Bloxham_Gubbins_Jackson_1989,Constable_Johnson_Lund_2000,Sabaka_Olsen_Langel_2002}).  This situation is similar to the problem of weather prediction on earth.  Therefore the techniques of data assimilation, which were developed in weather prediction, have been applied to the MHD equations in recent years (see, e.g., \cite{Bunge_Richards_Baumgardner_2002,Chandrasekar_Ridley_Bernstein_2007,Fournier_Eymin_Thierry_2007,Fournier_Hulot_Jault_Kuang_Tangborn_Gillet_Canet_Aubert_Lhuillier_2010,Groth_DeZeeuw_Gombosi_Powell_2000,Li_Jackson_Livermore_2014,Mendoza_DeMoor_Bernstein_2006,Scott_McKenzie_Longcope_2016,Sun_Tangborn_Kuang_2007,Teixeira_Ridley_Torres_Aguirre_Bernstein_2008}).  It has also been speculated in \cite{Adams_Stone_Zimmerman_Lathrop_2015} that data assimilation for magnetohydrodynamics may be useful in liquid sodium experiments modeling the Earth's core.

Data assimilation has been the subject of a very large body of work.  Classically, these techniques are based on linear quadratic estimation, also known as the Kalman Filter.  The Kalman Filter has the drawback of assuming that the underlying system and any corresponding observation models are linear.  It also assumes that measurement noise is Gaussian distributed.  This has been somewhat corrected via modifications, such as the Extended Kalman Filter and the Unscented Kalman Filter.  For more about the Kalman Filter and its modifications, see, e.g., \cite{Daley_1993_atmospheric_book,Kalnay_2003_DA_book,Law_Stuart_Zygalakis_2015_book}, and the references therein.  Recently, a promising new approach to data assimilation was pioneered by Azouani, Olson, and Titi in \cite{Azouani_Olson_Titi_2014,Azouani_Titi_2014} (see also \cite{Cao_Kevrekidis_Titi_2001,Hayden_Olson_Titi_2011,Olson_Titi_2003} for early ideas in this direction).  This new approach is based on feedback control at the PDE level.  The first works in this area assumed noise-free observations, but \cite{Bessaih_Olson_Titi_2015} adapted the method to the case of noisy data, and \cite{Foias_Mondaini_Titi_2016} adapted it to the case where measurements are obtained discretely in time and may be contaminated by systematic errors.   Computational experiments on this technique were carried out in the cases of the 2D Navier-Stokes equations \cite{Gesho_Olson_Titi_2015}, the 2D B\'enard convection equations \cite{Altaf_Titi_Knio_Zhao_Mc_Cabe_Hoteit_2015}, and the 1D Kuramoto-Sivashinsky equations \cite{Lunasin_Titi_2015,Larios_Pei_nonlinear_da_2017}.  In \cite{Larios_Pei_nonlinear_da_2017}, several nonlinear versions of this approach were proposed and studied. In addition to the results discussed here, a large amount of recent literature has built upon this idea; see, e.g., \cite{Albanez_Nussenzveig_Lopes_Titi_2016,Biswas_Martinez_2017,Farhat_Jolly_Titi_2015,Farhat_Lunasin_Titi_2016abridged,Farhat_Lunasin_Titi_2016benard,Farhat_Lunasin_Titi_2016_Charney,Farhat_Lunasin_Titi_2017_Horizontal,Foyash_Dzholli_Kravchenko_Titi_2014,GlattHoltz_Kukavica_Vicol_2014,Jolly_Martinez_Titi_2017,Jolly_Sadigov_Titi_2015,Markowich_Titi_Trabelsi_2016,Mondaini_Titi_2017}.  

In the present work, we adapt the approach of \cite{Azouani_Olson_Titi_2014,Azouani_Titi_2014,Farhat_Lunasin_Titi_2016abridged} to the 2D MHD equations.  In Theorem \ref{thm all}, we show that solutions of the feedback-controlled system converge exponentially in the $L^2$-norm to solutions of the MHD system when feedback control is applied to all variables (here, we use Els\"asser variables for simplicity).  This convergence holds under certain conditions on the spacing of the data and the weight given to the feedback control.  Moreover, in Theorems \ref{thm 1st} and \ref{thm v}, we establish {\it abridged data assimilation}, i.e., we show that feedback control need only be applied to a reduced set of the variables (horizontal variables or a single Els\"asser variable, respectively) to obtain exponential convergence, at the cost of more restrictive conditions on the data resolution $h$ and control weight $\mu$.  In Theorem \ref{thm H1 all}, we establish exponential convergence in the $H^1$-norm.  Next, in Theorem \ref{T2 thm 1}, we show that if one makes weaker assumptions on the data interpolation function, and if feedback control is applied only to horizontal variables, then exponential convergence in the $H^1$ norm holds as well. Finally, in Section \ref{sec:detinterp}, we establish a rigorous connection between data assimilation and the concept of determining quantities, first introduced in \cite{foias-prodi}, and further studied in \cite{ftdet,cjt,jt1,jt2,jt3}.

\subsection{Background on Data Assimilation}
We now describe the general idea of the data assimilation scheme we use for the 2D MHD equations, 
based on the idea of feedback control, that was developed by Azouni, Olson and Titi in \cite{Azouani_Olson_Titi_2014,Azouani_Titi_2014} in the context of the 2D Navier-Stokes equations.  In the study of a dynamical system in the form,
 $$ \dt Y = F(Y), $$
 subject to certain boundary conditions, one normally 
 tries to show that unique solutions will arise given 
 any initial value
 $$Y(0)=Y_0,$$
 in a certain space, and that the solution will change 
 in a continuous way with respect to a change in the 
 initial value.
 
 The problem arises in practice that the initial value
 may not be known exactly, but it may approximate the
 true initial value of a given observable, for example
 the temperature, which we'd like to predict the value 
 of in the future. The continuous dependence on initial
 data addresses this issue, in that if the initial
 approximation is close enough to the true value, then 
 the solution we obtain will accurately approximate the 
 true value of the observable for some period of time.
 However, usual theory shows that the length of time
 the approximation is guaranteed to be good is short,
 in that the error may grow exponentially in time.
 Also, the initial measurement may need to give a very
 close approximation to the true initial value, but
 in practice measurements may only be available on a 
 coarse grid, limiting the accuracy of the initial
 approximation and thus limiting both the accuracy the 
 solution can be guaranteed to have, as well as the 
 duration for which this accuracy can be guaranteed.
 
 Data assimilation is the method where, to compensate 
 for this lower bound on the accuracy of the 
 measured initial condition, measurements are taken 
 of the observable as time goes on 
 (over the same possibly coarse grid on which the initial 
 value is approximated) and fed back into
 the differential equation 
 (giving a different equation, called the data 
 assimilation equation) in such a way 
 that the solution will become a better approximation
 as time goes on. This gives us the accuracy we need
 to apply the continuous dependence on initial data 
 and say the prediction will be accurate for some 
 duration from that time onwards.
 
 The data assimilation algorithm (the way measurements
 are introduced to the differential equation) can take
 different forms, but the one we consider here was first
 introduced by Azouani, Olson, and Titi in \cite{Azouani_Olson_Titi_2014,Azouani_Titi_2014}. Given that the true value of the observable at time $t$ is $Y(t),$ then the data assimilation equation
 will be:
 \begin{align*}
 \dt \tilde{Y} &= F(\tilde{Y})+\mu(\oI(Y)-\oI(\tilde{Y}))\\ 
  &= F(\tilde{Y}) + \mu\oI(Y-\tilde{Y}),
 \end{align*}
 where  the second equality in the above equation follows because we'll assume the
 interpolant operator, $\oI$, is linear. Here, $\mu$ will be an adequately chosen tuning parameter. In addition,  we will assume that for all $u\in\sH$, $\oI$ satisfies 
one of the following:
 \begin{equation}\label{I1}
  \nl{u-\oI(u)}\leq\cI h\nh{u},
 \end{equation}
 or
 \begin{equation}\label{I2}
  \nl{u-\oI(u)}\leq\cIa h\nh{u} + \cIb h^2\nH{u}.
 \end{equation}
 Many relevant examples of operators satisfy one of these two conditions, including the projection onto the low modes, finite volume element operators, and nodal interpolant operators. For more information, see, e.g. \cite{Lunasin_Titi_2015,Foias_Titi_1991_Nonlinearity,Azouani_Olson_Titi_2014}.
 
\subsection{Background on the MHD equations}
 We consider the 2D MHD equations for a fluid and magnetic field under periodic boundary conditions and with zero space average. Let $u$, $b$, 
 and $p$ represent the fluid velocity, magnetic field, and 
 fluid pressure, respectively, and let the spatial domain be $[0,L]^2$.
 The system can be written as (see, e.g., \cite{Davidson_2001}): 
 \begin{align*}
 \pt u - \nu\ol{u} + \ob{u}{u} - \tfrac{1}{\rho_0\mu_0}\ob{b}{b} 	
	&= - \tfrac{1}{\rho_0}\og\left(p+\tfrac{1}{2\mu_0}\a b^2\right) + f_1,\\
 \pt b - \lambda\ol{b} + \ob{u}b - \ob{b}u 		&= g_1,\\
 \nabla\cdot b =0, \quad \nabla\cdot u &= 0.
 \end{align*}
 Here, $\nu>0$ is the kinematic fluid viscosity, $\rho_0$ is the fluid density, $\mu_0:=4\pi\times10^{-7} H/m$ is the permeability of free space, $\lambda=(\mu_0\sigma)^{-1}>0$ is the magnetic diffusivity,  and $\sigma$ is the electrical conductivity of the fluid.  We impose initial conditions $u(0,x,y) = u_0(x,y)$ and  $b(0,x,y) = b_0(x,y)$ in an appropriate function space, and allow for time-dependent forcing functions, denoted above by $f_1$ and $g_1$.

 Our analyses will have to take into account the amount of energy being added to the system by the forcing functions, so to this end we define the Grashof number, $G$, to be
 $$G:=\tfrac{8}{\lambda_1}\max\{\tfrac{1}{\nu^2},\tfrac{1}{\lambda^2}\}
 \limsup_{t\to\infty}\left(\max\left\{
	 \|f_1(t)\|_{L^2([0,L])},\quad
	\tfrac{1}{\sqrt{\rho_0\mu_0}}\|g_1(t)\|_{L^2([0,L])}
 \right\}\right).$$
where $\lambda_1:=\frac{4\pi^2}{L^2}$ is the smallest eigenvalue of the Stokes operator on the space of functions with space average zero on $[0,L]^2$ under periodic boundary conditions \cite{Foias_Manley_Rosa_Temam_2001}.

  Note that we have constructed $G$ to be dimensionless. We will also non-dimensionalize the system so that we can later reformulate it in terms of the Els\"asser variables. Let $U$ be a reference velocity and use $L$ as a reference length. We denote the dimensionless fluid Reynolds number and the dimensionless magnetic Reynolds number by $Re:=UL/\nu$ and $Rm:=UL/\lambda$, respectively.  
 In non-dimensional form, the system can be written as:
 \begin{subequations}\label{MHD}
 \begin{align}
 \pt u - \tfrac{1}{Re}\ol{u} + \ob{u}{u} - \ob{b}{b} 	
	&= - \og\PT + f_1,\label{sys MHD u}\\
 \pt b - \tfrac{1}{Rm}\ol{b} + \ob{u}b - \ob{b}u 		&= g_1,\label{sys MHD b}\\
 \nabla\cdot b =0, \quad  \nabla\cdot u &= 0.
 \end{align}
 \end{subequations}
with the initial conditions $u(0,x,y) = u_0(x,y)$ and $b(0,x,y) = b_0(x,y)$, and where $\PT$ is the (non-dimensionalized) sum of the fluid and magnetic pressures, and $u$, $b$, $u_0$, $b_0$, $f_1$, and $g_1$ have been replaced by their appropriate non-dimensional versions.  Note the bilinearity  in $(u,b)$ on the left-hand side of \eqref{sys MHD b} allows for the important fact that the four non-linear terms in \eqref{MHD} can be written with coefficients $\pm1$. 
 We will denote the non-dimensionalized spatial domain by 
 $$\Omega:=[0,1]^2\subset \mathbb{R}^2.$$

 Global existence and uniqueness of solutions to \eqref{MHD vw} was proven in \cite{Duvaut_Lions_1972} and \cite{Sermange_Temam1983}. For a derivation and physical discussion of the MHD equations, see, e.g., \cite{Chandrasekhar_1961}. For an overview of the classical and recent mathematical results pertaining to the MHD equations, see, e.g., \cite{Duvaut_Lions_1972,Davidson_2001}.


\section{preliminaries}

\noindent
In this section, we briefly lay out some notation, discuss some of the standard results and inequalities we use, and give the specific equations we will discuss.
~\\
\paragraph{\bf Notation}~\\
For a matrix $A$, we denote $\a A^2 := \sum_{i,j}|A_{i,j}|^2$.  
We denote the standard $L^2$ inner-product and norm by $\ip u v := \Ip u v$ and 
$\nl u := \left(\I{\a{u}^2}\right)^\frac{1}{2}$, respectively (note that the integral is taken over the non-dimensionalized domain, $\Omega$, so $\nl{u}$ has the same units as $u$).  We also denote
$\n{u}_{H^1}:=\nh{u}$, which is equivalent to the standard $H^1$ norm, due to the Poincar\'e inequality \eqref{Poincare}.
~\\
\paragraph{\bf Inequalities}~\\
We recall some standard inequalities.  Here $\epsilon > 0$, $a,b \geq 0$, and $u,v,$ and $w$ are divergence-free periodic functions,
with sufficient regularity to make all the norms involved finite.

We will frequently use the following forms of Young's inequality and H\"older's inequality:
\begin{align}
 \label{Young}
  ab&\leq \frac{\epsilon}{2} a^2 + \frac{1}{2\epsilon} b^2\\
 \label{Holder}
  \a{\I {u v w}} &\leq \nl{u} \nL{v} \nL{w} 
\end{align}

We also recall the following version of Poincar\'e's inequality, valid for periodic functions with zero space average on $\Omega$:
\begin{align}
 \label{Poincare}
  \nh{u}\geq 2\pi\nl{u} 
\end{align}

The following inequality due to Ladyzhenskaya will be used to 
bound the nonlinear terms for the cases where we have measurements
on all the components and when we only measure one Els\"asser variable:
\begin{align}
 \label{Lady}
  \nL{u}^2\leq \cL \nl{u} \nh{u} 
\end{align}

The next two inequalities are extensions of the Brezis-Gallouet and are
due to Titi \cite{Titi1987}. They will be necessary to bound the nonlinear 
terms in the case of measuring only one component of the reference 
velocity and magnetic fields:
\begin{align}
 \label{Brezis}
  \a{\I {u \pxi v w}} \leq \cB \nh{u}\nh{v}\nl{w}
  \left(1 + \ln\left(\frac{\nh{w}}{2\pi
	\nl{w}}\right)\right)^{1/2}, \\
 \label{Titi-Brezis}
    \a{\I {u \pxi v \ol w}} \leq \cT \nh{u}\nh{v}\nH{w}
    \left(1 + \ln\left(\frac{\nH{z}}{2\pi
		\nh{z}}\right)\right)^{1/2},
\end{align}
where in \eqref{Titi-Brezis}, $z$ can be $u$ or $v$.

The following generalization of the Gr\"onwall Lemma will be useful, which was first shown by Foias et al. in \cite{Foias_Manley_Temam_Treve_1983}. For a proof of an even more general version due to Jones and Titi, see \cite{Foias_Manley_Rosa_Temam_2001}.

\begin{proposition}[Generalized Gronwall Inequality]\label{gronwall}
Let $\psi:[0,\infty)\to\sR$ be a locally integrable function such that for some 
	$T>0$ the following two conditions hold:
 \begin{subequations}
 \begin{equation}\label{gronwall inf}
 \liminf_{t\to\infty}\frac{1}{T}\int_{t}^{t+T}\psi(s)ds>0,
 \end{equation} 
 \begin{equation}\label{gronwall sup}
 \limsup_{t\to\infty}\frac{1}{T}\int_{t}^{t+T}\psi^-(s)ds<\infty,
 \end{equation}
where $\psi^-(t):=\max\{0,-\psi(t)\}.$ Then if $Y:[0,\infty)\to[0,\infty)$ is 
	absolutely continuous and for almost all $t,$
 \begin{equation}\label{gronwall dif}
  \dt Y + \psi Y \leq \phi,
 \end{equation}
 \end{subequations}
	where $\phi(t) \to 0$ as $t\to\infty$, then $Y(t)\to 0$ as well. 
Furthermore, if $\phi\equiv 0$ then $Y(t)\to 0$ exponentially as $t\to\infty.$
\end{proposition}

Next, in order to simplify our calculations we will reformulate the MHD equations
in terms of new variables which we call $v$ and $w$, in such
a way as to symmetrize the system.

We assume, without loss of generality, that 
$\tfrac{1}{Re} \geq \tfrac{1}{Rm}$, and denote the Els\"asser variables \cite{Elsasser_1950} by $v=u+b$ and $w=u-b$ 
(if $\tfrac{1}{Re} < \tfrac{1}{Rm}$ then we would denote $w=b-u$ and proceed similarly).

Then we can derive evolution equations for $v$ and $w$ by considering 
both the sum and difference of \eqref{sys MHD u} and \eqref{sys MHD b}
and obtain the following system: 
\begin{system}
\begin{subequations}\label{MHD vw}
\begin{align}
\label{MHD_v}
 \pt v - \alpha \ol v - \beta \ol w + \ob w v &= -\og \PT + f,\\
 \label{MHD_w}
 \pt w - \alpha \ol w - \beta \ol v + \ob v w &= -\og \PT + g,\\
 \label{MHD_vw_div}
 \odiv v = 0, \quad \odiv w &= 0,
 \end{align}
\end{subequations}
subject to the initial conditions $v(0)=v_0:=u_0+b_0$ and $w(0)=w_0:=u_0-b_0$. 
\end{system}
Here we relabeled the forcing terms as
$f:=f_1+g_1$ and $g:=f_1-g_1$, and we denote $\alpha:=\half(\tfrac{1}{Re}+\tfrac{1}{Rm})$ and 
$\beta:=\half(\tfrac{1}{Re}-\tfrac{1}{Rm})$.
It will be important to note that $\alpha-\beta=\tfrac{1}{Rm}>0$ and that
$\alpha>0$ and $\beta\geq 0$ (this last inequality is true by 
the assumption that $\tfrac{1}{Re}\geq \tfrac{1}{Rm}$, however if
$\tfrac{1}{Re}<\tfrac{1}{Rm}$ then we would arrive at the above system
except with a different sign on the pressure, and $\beta=\tfrac{1}{Rm}-\tfrac{1}{Re}$, so still we have
$\beta\geq 0$, and in general we will have $\alpha-\beta = \min\{\tfrac{1}{Re},\tfrac{1}{Rm}\}$).

We note here that $G$ can be expressed in terms of the forcing functions for the reformulated system:
$$G=\tfrac{\max\{Re^2,Rm^2\}}{\pi^2}\limsup_{t\to\infty}
	\left(\max\lbrace\nl{f(t)+g(t)},\nl{f(t)-g(t)}\rbrace\right),$$
hence,
$$G\geq\tfrac{1}{\pi^2(\alpha-\beta)^2}\limsup_{t\to\infty}\left(\max\lbrace\nl{f(t)}, \nl{g(t)}\rbrace\right),$$

Now, we describe the data assimilation algorithms studied in this paper. 
Following the ideas of \cite{Azouani_Olson_Titi_2014,Azouani_Titi_2014} we incorporate
measurements obtained from a fixed reference solution 
(of which we want to predict future values) through a damping term.
This will ``steer'' the data assimilation solutions to the reference
solution exponentially in time. In what sense we will have convergence
depends on the type of interpolant $\oI$ with which we take measurements.

The results are separated by the type of interpolant considered and 
by which measurements are recorded. We 
frame our results in terms of the Els\"asser variables, not in
terms of $u$ and $b$. Also, we consider algorithms which
require measurements taken only on the first components, $u_1$ and $b_1$
(which is the same as measuring $v_1$ and $w_1$), by measuring all 
the components of $u$ and $b$, or by measuring either the sum $u+b$
or the difference $u-b$ only. 


In the following, let $(v,w)$ be a fixed solution of \eqref{MHD vw}, and
we denote the data assimilation variables by $\vt$ and $\wt$, which
will approximate $v$ and $w$ respectively. $\oI$ may satisfy either
\eqref{I1} or \eqref{I2}, and we will analyze each case separately.

First, we have the following algorithm which utilizes measurements
taken on all components (so measuring $u$ and $b$):
\begin{system}
\begin{subequations}\label{MHD data all}
\begin{align}
 \pt \vt - \alpha \ol \vt - \beta \ol \wt + \ob \wt \vt &= 
	-\og \tilde{\PT} + f + \mu \oI(v-\vt)\\
 \pt \wt - \alpha \ol \wt - \beta \ol \vt + \ob \vt \wt &= 
	-\og \tilde{\PT} + g + \mu \oI(w-\wt)\\
 \odiv \vt = 0, \quad \odiv \wt &= 0,
\end{align}
\end{subequations}
subject to the initial conditions $\vt(0)\equiv \wt(0)\equiv 0$.
\end{system}

Next, using measurements only on the first components of $v$ and $w$ (which is equivalent to measuring $u_1$ and $b_1$):
\begin{system}\label{sys data 1}
\begin{subequations}\label{MHD data 1 comp}
\begin{align}
 \pt \vt - \alpha \ol \vt - \beta \ol \wt + \ob \wt \vt &= 
	-\og \tilde{\PT} + f + \mu \oI(v_1-\vt_1)e_1\\
 \pt \wt - \alpha \ol \wt - \beta \ol \vt + \ob \vt \wt &= 
	-\og \tilde{\PT} + g + \mu \oI(w_1-\wt_1)e_1\\
 \odiv \vt = 0, \quad \odiv \wt &= 0,
\end{align}
\end{subequations}
subject to the initial conditions $\vt(0)\equiv \wt(0)\equiv 0$.
\end{system}

Finally, only taking measurements on $v$ (which would in practice still require recording measurements on both $u$ and $b$):
\begin{system}\label{sys data v}
\begin{subequations}\label{MHD data v}
\begin{align}
 \pt \vt - \alpha \ol \vt - \beta \ol \wt + \ob \wt \vt &= 
	-\og \tilde{\PT} + f + \mu \oI(v-\vt)\\
 \pt \wt - \alpha \ol \wt - \beta \ol \vt + \ob \vt \wt &= -\og \tilde{\PT} + g\\
 \odiv \vt = 0, \quad \odiv \wt &= 0,
\end{align}
\end{subequations}
subject to the initial conditions $\vt(0)\equiv \wt(0)\equiv 0$.
\end{system}

\begin{remark}
Although we chose to consider taking measurements on the first components of $v$ and $w$ in System~\ref{sys data 1}, we could instead use the second components with no substantial differences. Likewise, in System~\ref{sys data v} we could also consider taking measurements on $w$ and we would obtain similar results.
\end{remark}

\begin{remark}\label{remark_init_data}
In the above we chose to make the initial conditions $0$, but in fact the initial conditions may be chosen essentially arbitrarily, albeit in accordance with the existence theorems. Theorem~\ref{T2 thm 1} additionally requires that the initial conditions satisfy an upper bound of the form \eqref{sup bound}. 
\end{remark}

\begin{remark}
Here we first constructed the Els\"asser variables from the original variables $u$ and $b$ after nondimensionalizing, and then proceeded to define the various data assimilation algorithms and variables. However, since the transformations were linear, if we were to define each data assimilation algorithm using the original variables, in the process defining data assimilation variables $\ut$ and $\tilde{b}$, and then nondimensionalize and change to the Els\"asser variables, we would arrive at the same systems above. So, all our results apply to the corresponding algorithms formulated in terms of the original variables.\\
Note also that although the results are framed in terms of the Els\"asser variables, by the triangle inequality convergence of $\tilde{v}$ to $v$ and $\tilde{w}$ to $w$ implies convergence of $\tilde{u}$ and $\tilde{b}$ to $u$ and $b$ respectively.
\end{remark}

We define weak solutions for all the systems mentioned in the distributional sense in the usual way. See \cite{Sermange_Temam1983} for a precise definition in the case of \eqref{MHD} (the other systems are similar). In addition to being a weak solution, we say $(v,w)$ (or $(\vt,\wt)$) is a global strong solution of \eqref{MHD vw} (or \eqref{MHD data all}, \eqref{MHD data 1 comp}, or \eqref{MHD data v}) if 
$$ v,w\in\sL(0,T;H^2)\cap L^\infty(0,T;\sH),\quad \forall T>0.$$
In \cite{Sermange_Temam1983}, it was shown that if $\esssup_{[0,\infty)} \nl{f_1} < \infty$ and $u_0,b_0 \in\sH$, then there exists a unique global strong solution to \eqref{MHD} (which can be transformed to a solution of \eqref{MHD vw}). Therefore, we will be assuming that, in addition to being space periodic and divergence free,
\begin{equation*}
	\esssup_{[0,\infty)} \max\{\nl{f},\nl{g}\} < \infty \quad\text{and}\quad
		\nh{u_0},\nh{b_0}<\infty. \hfill
\end{equation*}
The proofs of the corresponding existence and uniqueness results for Systems 2.3-2.5
are similar, and are omitted. We only state and prove the corresponding convergence results.

\section{Statements of the Results}

\subsection{Results for Type 1 Interpolants}
\begin{theorem}\label{thm all}
 Let $(v,w)$ be a strong solution of \eqref{MHD vw} which at time $t=0$ has evolved enough so that Proposition~\ref{prop upper bounds} holds with $t_0=0$. Let $\oI$ satisfy \eqref{I1}, where
\begin{align*}
h \leq\cI^{-1}(\alpha-\beta)^{\half}\mu^{-\half}, \text{ and }\quad
     \mu > \tfrac{\pi^2(\cL^4+(\alpha-\beta)^4)}{\alpha-\beta}G^{2}
\end{align*} 
(so $h\sim G^{-1}$).
 Then there is a unique strong solution, $(\vt,\wt)$, of \eqref{MHD data all} corresponding to $(v,w)$ which exists globally in time, and furthermore $\nl{v(t)-\vt(t)} + \nl{w(t)-\wt(t)} \to 0$ exponentially as $t\to\infty$.
\end{theorem}

\begin{theorem}\label{thm 1st}
 Let $(v,w)$ be a strong solution of \eqref{MHD vw} which at time $t=0$ has evolved enough so that Proposition~\ref{prop upper bounds} holds with $t_0=0$. Let $\oI$ satisfy \eqref{I1}, where
 \begin{align*}
h \leq\cI^{-1}(\alpha-\beta)^{\half}\mu^{-\half}, 
\text{ and }\quad
\mu > 32\pi^2c^2\nua\left(\tilde{c}+2\ln G + \C G^4\right)G^2
 \end{align*} 
 (so $h \sim G^{-3}$).
  Then there is a unique strong solution, $(\vt,\wt)$, of \eqref{MHD data 1 comp} corresponding to $(v,w)$ which exists globally in time, and furthermore $\nl{v(t)-\vt(t)} + \nl{w(t)-\wt(t)} \to 0$ exponentially as $t\to\infty$.
\end{theorem}

\begin{theorem}\label{thm v}
 Let $(v,w)$ be a strong solution of \eqref{MHD vw} which at time $t=0$ has evolved enough so that Proposition~\ref{prop upper bounds} holds with $t_0=0$. Let $\oI$ satisfy \eqref{I1}, where
 \begin{align*}
 h \leq\cI^{-1}(\alpha-\beta)^{\half}\mu^{-\half},
 \text{ and }\quad
\mu>\frac{\pi^2\cL^4 G^2(4+\nua^2 G^2)^2}{16\nua}
\end{align*}
 (so $h\sim G^{-3}$).
 Then there is a unique strong solution, $(\vt,\wt)$, of \eqref{MHD data v} corresponding to $(v,w)$ which exists globally in time, and furthermore $\nl{v(t)-\vt(t)} + \nl{w(t)-\wt(t)} \to 0$ exponentially as $t\to\infty$.
\end{theorem}

In the next three theorems, by using the $\sL$ convergence results we just established, we show that solutions of \eqref{MHD data all}, \eqref{MHD data 1 comp}, and \eqref{MHD data v} will converge exponentially in time to the reference solution in the stronger topology of the $\sH$-norm.
\begin{theorem}\label{thm H1 all}
Let $(v,w)$ be a strong solution of \eqref{MHD vw}
which at time $t=0$ has evolved enough so that 
Proposition~\ref{prop upper bounds} holds with $t_0=0$.
Let $\oI$ satisfy \eqref{I1}, where 
\begin{align*}
	h < (2\sqrt{2}\cI)^{-1}\nua^{\half}\mu^{-\half}, \text{ and }\quad
     \mu > \frac{\pi^2(\cL^4+(\alpha-\beta)^4)}{\alpha-\beta}G^{2}
\end{align*} 
(so $h \sim G^{-1}$).
Then there is a unique strong solution, $(\vt,\wt)$,
of \eqref{MHD data all} corresponding to $(v,w)$
which exists globally in time, and furthermore
$\Vert v(t) - \vt(t)\Vert_{H^1} + \Vert w(t) - \wt(t)\Vert_{H^1} \to 0$ exponentially as $t\to\infty$.
\end{theorem}
\begin{theorem}\label{thm H1 1st}
Let $(v,w)$ be a strong solution of \eqref{MHD vw}
which at time $t=0$ has evolved enough so that 
Proposition~\ref{prop upper bounds} holds with $t_0=0$.
Let $\oI$ satisfy \eqref{I1}, where 
\begin{align*}
	h < (2\sqrt{2}\cI)^{-1}\nua^{\half}\mu^{-\half}, \text{ and }\quad
    \mu > 32\pi^2c^2\nua\left(\tilde{c}+2\ln G + \C G^4\right)G^2
\end{align*} 
(so $h \sim G^{-3}$).
Then there is a unique strong solution, $(\vt,\wt)$,
of \eqref{MHD data 1 comp} corresponding to $(v,w)$
which exists globally in time, and furthermore
$\Vert v(t) - \vt(t)\Vert_{H^1} + \Vert w(t) - \wt(t)\Vert_{H^1} \to 0$ exponentially as $t\to\infty$.
\end{theorem} 
\begin{theorem}\label{thm H1 v}
Let $(v,w)$ be a strong solution of \eqref{MHD vw}
which at time $t=0$ has evolved enough so that 
Proposition~\ref{prop upper bounds} holds with $t_0=0$.
Let $\oI$ satisfy \eqref{I1}, where 
\begin{align*}
	h < (2\sqrt{2}\cI)^{-1}\nua^{\half}\mu^{-\half}, \text{ and }\quad
     \mu>\frac{\pi^2\cL^4 G^2(4+\nua^2 G^2)^2}{16\nua}
\end{align*} 
(so $h \sim G^{-3}$).
Then there is a unique strong solution, $(\vt,\wt)$,
of \eqref{MHD data v} corresponding to $(v,w)$
which exists globally in time, and furthermore
$\Vert v(t) - \vt(t)\Vert_{H^1} + \Vert w(t) - \wt(t)\Vert_{H^1} \to 0$ exponentially as $t\to\infty$.
\end{theorem} 

\begin{remark}
Observing the Poincar\'e inequality, the results of Theorems~\ref{thm H1 all}-\ref{thm H1 v} seem to imply those of Theorems~\ref{thm all}-\ref{thm v}, but the spatial resolution is required to be slightly finer for the $\sH$ results. Also, based on our analysis, there may be a longer period of time that must pass before exponential convergence is observed in the $\sH$-norm than in the $\sL$-norm (see the estimates in \eqref{H1_via_L2_One} and \eqref{H1_via_L2_Two}). However, we point out that in computational results regarding data assimilation in the context of the one-dimensional Kuramoto-Sivasinsky equation, convergence times for both norms are almost identical (c.f. \cite{Larios_Pei_nonlinear_da_2017} for more details). 
\
\end{remark}
 
\subsection{Results for Type 2 Interpolants}
 \begin{theorem}\label{T2 thm 1}
  Let $(v,w)$ be a strong solution of \eqref{MHD vw},
  which at time $t=0$ has evolved enough so that 
  Proposition~\ref{prop upper bounds} holds with $t_0=0$.
  Then $h\sim G^{-6}e^{-\C G^4}$ and $\mu \sim G^{12}e^{2\C G^4}$ can be chosen so that if $\oI$ satisfies \eqref{I2} then there is a unique strong solution $(\vt,\wt)$ 
  of \eqref{MHD data 1 comp} corresponding to $(v,w)$
  which exists globally in time, and $\Vert v(t)-\vt(t)\Vert_{H^1} + \Vert w(t)-\wt(t)\Vert_{H^1} \to 0$ exponentially as $t\to\infty$.
 \end{theorem}
 
 \begin{remark}
  Similar theorems hold for the cases of measurements on all variables and one Els\"asser variable (although not as direct corollaries, since the dynamical systems involved are slightly different).  However, in the case of measuring all variables we do not find much improvement in the restrictions on $h$ and $\mu$.
 \end{remark}
 ~
  
\subsection{Determining Interpolants}~ \label{sec:detinterp}\\

 In order to prove that there are finitely many (say N) determining modes, 
 one needs to show that if $(v^{(1)},w^{(1)})$ and $(v^{(2)},w^{(2)})$ 
 are different solutions of \eqref{MHD vw} with possibly different forcing 
 terms and initial data, then knowledge that $\nl{P_N(v^{(1)},w^{(1)}) - P_N(v^{(2)},w^{(2)})}\to~0$ is 
 sufficient to conclude that 
 $\nl{(v^{(1)},w^{(1)}) - (v^{(2)},w^{(2)})}\to 0$, 
 where $P_N$ denotes the projection onto
 the modes with magnitude at most $N$. In general, we replace $P_N$ 
 by a different operator, say $\oI$, and ask the question of
 whether the knowledge inherent in $\oI$ is ``determining".

 In the following theorems, we show that the data assimilation results
 we have obtained can be adapted to show that the interpolant operators,
 $\oI$, are determining. We do this by first generalizing the convergence results we developed in the previous theorems to allow for the evolution 
 equations of the reference solution and the data assimilation solution
 to have different forcing terms, which converge in $\sL$ as $t\to\infty$,
 at the cost of losing the exponential rate of convergence of the solutions.
 We also allow for the reference solution to be perturbed by a function 
 which decays in $\sL$.
 
 We illustrate the ideas for the algorithm studied in Theorem~\ref{thm all},
 i.e. with measurements taken on all variables and for $\oI$ satisfying 
 \eqref{I1}, but the results can be obtained for all the other cases as 
 well. So, we can show that operators which satisfy \eqref{I1} or 
 \eqref{I2} and use measurements on $(v,w),(v_1,w_1),$ or $v$,
 are determining in the sense of convergence in $\sL$ and $\sH$.
 
 \begin{theorem}\label{thm gen. DA}
 Let $\oI$ satisfy \eqref{I1} and let $(v,w)$ be a reference solution of
 \eqref{MHD vw}. Then if $\mu$ and $h$ satisfy the hypotheses of 
 Theorem~\ref{thm all}, and if 
 $\nl{\delta^{(1)}(t)},\nl{\delta^{(2)}(t)}\to 0$ and
 $\nl{\oI(\epsilon^{(1)}(t))},\nl{\oI(\epsilon^{(2)}(t))}\to 0$ 
 as $t\to\infty$, 
 there are unique $\vt,\wt$ and $\tilde{\PT}$ which satisfy the 
 following modified version of \eqref{MHD data all}:
 \begin{system}
  \begin{subequations}\label{gen MHD data all}
  \begin{equation}
  \pt \vt - \alpha \ol \vt + \beta \ol \wt + \ob \wt \vt = 
 	-\og \tilde{\PT} + f + \delta^{(1)} + \mu \oI(v + \epsilon^{(1)} -\vt)
  \end{equation}
  \begin{equation}
  \pt \wt - \alpha \ol \wt + \beta \ol \vt + \ob \vt \wt = 
 	-\og \tilde{\PT} + g + \delta^{(2)} + \mu \oI(w + \epsilon^{(2)}-\wt)
  \end{equation}
  \begin{equation}
  \odiv \vt = 0,
  \end{equation}
  \begin{equation}
  \odiv \wt = 0,
  \end{equation}
  \end{subequations}
  subject to the initial conditions $\vt(0)\equiv 0,\wt(0)\equiv 0$,
 \end{system}
 \noindent and furthermore, $\nl{v-\vt}, \nl{w-\wt}\to 0$ as $t\to\infty$.
 \end{theorem}
 
 In the next theorem we illustrate the result that if an interpolant $\oI$
 satisfies the conditions for the generalized data assimilation theorem
 then $\oI$ is determining, for the case of the generalized version of
 Theorem~\ref{thm all}.
 Note that the projection onto the low modes, $P_N$, is an example of 
 an interpolant operator $\oI$ for which the theorem applies, provided 
 that $h:=\frac{1}{N}\lesssim G^{-1}$. Hence, the following theorem shows 
 that there are finitely many determining modes for instance.
  
 \begin{theorem}\label{thm det interp}
  Let $(v^{(1)},w^{(1)})$ and $(v^{(2)},w^{(2)})$ be solutions of \eqref{MHD vw} with forcing terms
  $f^{(1)},g^{(1)}$ and $f^{(2)},g^{(2)}$ respectively,
  and suppose that $\nl{f^{(1)}-f^{(2)}},\nl{g^{(1)}-g^{(2)}}\to 0$.\\
  Let $\oI$ satisfy \eqref{I1} where\\
  \vspace{-.2in}
  \begin{align*}
  & h < \frac{\alpha-\beta}{\pi\cI\sqrt{\cL^4+(\alpha-\beta)^4}}\, G^{-1},
  \text{ and }&\\
  & G:= 
   \frac{1}{\pi^2(\alpha-\beta)^2}\limsup_{t\to\infty}\left(\max\lbrace\nl{f^{(1)}(t)}, \nl{g^{(1)}(t)}\rbrace\right)
  	=\frac{1}{\pi^2\nua^2}\limsup_{t\to\infty}\left(\max\lbrace\nl{f^{(2)}(t)},
  	\nl{g^{(2)}(t)}\rbrace\right),
  \end{align*}
  \vspace{-.1in}\\
  and suppose that 
  $\nl{\oI(v^{(1)}(t)-v^{(2)}(t))},\nl{\oI{(w^{(1)}(t)-w^{(2)}(t))}}\to~0$
  as $t\to\infty$.\\
  Then $\nl{v^{(1)}(t)-v^{(2)}(t)},\nl{w^{(1)}(t)-w^{(2)}(t)}\to 0$
  as well.
 \end{theorem}
 

\section{Proofs of the Results} \label{proofs}
Before we get to the main proofs, we first state the following bounds for the reference solution to the MHD system. 
Moreover, we prove \eqref{int bound}, which follows standard arguments from the Navier-Stokes theory (see, e.g., \cite{Constantin_Foias_1988,Temam_2001_Th_Num}).  
The proofs of \eqref{sup bound} and \eqref{sup lap bound} can be obtained by modifying the corresponding proofs from the Navier-Stokes theory in a similar way (see, e.g. \cite{Duvaut_Lions_1972,Sermange_Temam1983} for more details on \eqref{sup bound} and the appendix of \cite{Farhat_Lunasin_Titi_2016abridged} for \eqref{sup lap bound}).
\begin{proposition}[Upper Bounds on Solutions of the MHD]
	\label{prop upper bounds}
Let $(v,w)$ be a solution of \eqref{MHD vw}. Then there is a 
$t_0>0$ and constants $\cM>0$ and $\C=\tfrac{81}{4}\cL^8$ such that for all $t\geq t_0$ and any $T>0$,
\begin{equation}\label{int bound}
\int_{t}^{t+T}\left(\nh{v(s)}^2 + \nh{w(s)}^2\right)\, ds
\leq \cupa G^2,
\end{equation}
\begin{equation}\label{sup bound}
\nh{v(t)}^2 + \nh{w(t)}^2 \leq \cupb G^2 e^{\C G^4}.
\end{equation}
\begin{equation}\label{sup lap bound}
 \nH{v(t)}^2 + \nH{w(t)}^2 \leq \cupT G^2
 	\left(1+\left(1+G^2e^{\C G^4}\right)\left(1+e^{\C G^4}+G^4e^{\C G^4}\right)\right).
\end{equation}
\end{proposition}
\begin{proof}[Proof of \eqref{int bound}]
	See the appendix.
\end{proof}

\subsection{Proofs of $\sL$ Convergence Results with Type 1 Interpolants}~

Before we get to the proofs of the main theorems, we first collect the various estimates needed for the bilinear term in the following lemma.
\begin{lemma}\label{lemma nonlinear term est}
Let $u,v,w\in \sH$ be divergence free. Then the following inequalities hold for 
any $\epsilon,\delta>0:$
 \begin{multline}\label{nonlin est 2}
  (a)\quad\a{\Ib u v w}\leq \frac{\cL\delta}{4}\nh{u}^2 
  + \half[\epsilon]\nh{w}^2 +  \frac{\cL\delta}{4}\nh{v}^2\nl{u}^2 
  + \frac{\cL^2}{8\epsilon\delta^2}\nh{v}^2\nl{w}^2,
 \end{multline} 
 \hspace{.5\textwidth} or
 \begin{multline}\label{nonlin est 2.1}
   \,\,\quad\quad\a{\Ib u v w}\leq \frac{\cL\delta}{4}\nh{w}^2 
   + \half[\epsilon]\nh{u}^2 +  \frac{\cL\delta}{4}\nh{v}^2\nl{w}^2 
   + \frac{\cL^2}{8\epsilon\delta^2}\nh{v}^2\nl{u}^2,
  \end{multline}
 \begin{multline}\label{nonlin est 1}
  (b)\quad\a{\Ib u v w}\leq c\delta\nh{u}^2 + c\delta\nh{w}^2 +
  \frac{c}{\delta}\nh{v}^2\left(\nl{u_1}^2 + \nl{w_1}^2\right)\\
  +\frac{c}{\delta}\nh{v}^2\left(1+\ln\frac{\nh{u_1}}{2\pi
		\nl{u_1}}\right)\nl{u_1}^2
  +\frac{c}{\delta}\nh{v}^2\left(1+\ln\frac{\nh{w_1}}{2\pi
		\nl{w_1}}\right)\nl{w_1}^2.
 \end{multline}
\end{lemma}
\begin{proof}
	See the appendix.
\end{proof}

The following lemma will be used in our analyses of the algorithms using measurements on only the first components of the reference solutions, where we will need to make use of \eqref{Brezis}, \eqref{Titi-Brezis}, or \eqref{nonlin est 1}. The proof is elementary, and therefore omitted.
\begin{lemma}\label{lemma log min}
Let $\phi(r) = r - \gamma(1+\ln(r)),$ for some $\gamma>0.$ Then $\forall r\geq1$,
 \begin{equation*}
 \phi(r)\geq -\gamma\ln(\gamma).
 \end{equation*}
\end{lemma}


\subsubsection*[All Components L2]{
}
\begin{proof}[\bf Proof of Theorem \ref{thm all}]~\\
 Let $\eta = v - \vt$ and $\zeta = w - \wt.$\\
 
 Then $\eta$ satisfies:
  $$ \pt\eta - \alpha\ol\eta - \beta\ol\zeta + \ob{w}{v} - \ob{\wt}{\vt} 
     = -\og(\PT - \tilde{\PT}) - \mu\oI(\eta). $$
  Using the fact that 
  $\ob{w}{v} - \ob{\wt}{\vt} = \ob{\zeta}{v} + \ob{\wt}{\eta}$
  we write:
  $$ \pt\eta - \alpha\ol\eta - \beta\ol\zeta + \ob{\zeta}{v} + \ob{\wt}{\eta} 
      = -\og(\PT - \tilde{\PT}) - \mu\oI(\eta). $$
  Taking the inner product with $\eta$ we obtain:
  $$ \half\dt\nl\eta^2 + \alpha\nh\eta^2 + \beta\ip{\og\zeta}{\og\eta} 
 	+ \ib{\zeta}{v}{\eta}
    = -\ip{\og(\PT-\tilde{\PT})}{\eta} - \mu\ip{\oI(\eta)}{\eta}.
  $$
  Now, by the divergence free condition,
  $$ -\ip{\og(\PT-\tilde{\PT})}{\eta} := -\Ip{\og(\PT-\tilde{\PT})}{\eta} 
     =\Ip{(\PT-\tilde{\PT})}{(\og\cdot\eta)} = 0.
  $$
  By applying Cauchy-Schwarz inequality and \eqref{Young},
  $$
    \a{\beta\ip{\og\zeta}{\og\eta}} 
 	\leq \half[\beta]\nh{\eta}^2 + \half[\beta]\nh{\zeta}^2,
  $$
  and by rewriting 
  $\ip{\oI(\eta)}{\eta} 
 	= \ip{\oI(\eta)-\eta}{\eta} + \ip{\eta}{\eta},$
  we have:
  $$ - \mu\ip{\oI(\eta)}{\eta} = 
     -\mu\ip{\oI(\eta)-\eta}{\eta} - \mu\nl{\eta}^2.
  $$
  Thus, we obtain:
  \begin{align*}
    \half\dt\nl\eta^2 + \left(\alpha-\half[\beta]\right)\nh\eta^2 
 		- \half[\beta]\nh\zeta^2 + \mu\nl{\eta}^2
     &\leq -\ib{\zeta}{v}{\eta} - \mu\ip{\oI(\eta)-\eta}{\eta}&\\
     &\leq \a{\ib{\zeta}{v}{\eta}} + \mu\a{\ip{\oI(\eta)-\eta}{\eta}}&\\
     &\leq \a{\ib{\zeta}{v}{\eta}} + \mu\nl{\oI(\eta)-\eta}\nl{\eta}&\\
     &\leq \a{\ib{\zeta}{v}{\eta}} + \mu\cI h\nh{\eta}\nl{\eta}&\\
     &\leq \a{\ib{\zeta}{v}{\eta}} + \half[\mu\cI^2 h^2]\nh{\eta}^2 
 		+ \half[\mu]\nl{\eta}^2,&\\
  \end{align*}
  where in the last three lines we used Cauchy-Schwarz inequality, the definition of $\oI$, 
 and Young's inequality. This leaves us with:
  \begin{equation*}
   \half\dt\nl\eta^2 + \left(\alpha-\half[\beta]\right)\nh\eta^2 
 	- \half[\beta]\nh\zeta^2 
   + \half[\mu]\nl{\eta}^2 \leq \a{\ib{\zeta}{v}{\eta}} 
 	+ \half[\mu\cI^2 h^2]\nh{\eta}^2.
  \end{equation*}
  Proceeding the same way for $\zeta$, we have the following
  equations:
  \begin{equation}\label{before non all e}
   \half\dt\nl\eta^2 + \left(\alpha-\half[\beta]-\half[\mu\cI^2 h^2]\right)\nh\eta^2 
   - \half[\beta]\nh\zeta^2 + \half[\mu]\nl{\eta}^2 \leq \a{\Ib{\zeta}{v}{\eta}},
  \end{equation}
  \begin{equation}\label{before non all z}
   \half\dt\nl\zeta^2 + \left(\alpha-\half[\beta]-\half[\mu\cI^2 h^2]\right)\nh\zeta^2 
   - \half[\beta]\nh\eta^2 + \half[\mu]\nl{\zeta}^2 \leq \a{\Ib{\eta}{w}{\zeta}}.
  \end{equation}
  We estimate the integrals in these equations using \eqref{nonlin est 2},
  with $\epsilon = \half[\alpha-\beta]$ and $\delta = \frac{\alpha-\beta}{\cL}$, and obtain
  \begin{multline}\label{after non all e}
   \half\dt\nl\eta^2 + \left(\alpha-\half[\beta]-\half[\mu\cI^2 h^2]-\frac{\alpha-\beta}{4}\right)\nh\eta^2 
   +\left(-\half[\beta]-\frac{\alpha-\beta}{4}\right)\nh\zeta^2 \\
   + \left(\half[\mu]-\frac{\cL^4}{4(\alpha-\beta)^3}\nh{v}^2\right)\nl{\eta}^2 
   + \left(-\frac{\alpha-\beta}{4}\nh{v}^2\right)\nl{\zeta}^2 \leq 0,
  \end{multline}
  \begin{multline}\label{after non all z}
   \half\dt\nl\zeta^2 + \left(\alpha-\half[\beta]-\half[\mu\cI^2 h^2]-\frac{\alpha-\beta}{4}\right)\nh\zeta^2 
   +\left(-\half[\beta]-\frac{\alpha-\beta}{4}\right)\nh\eta^2 \\
   + \left(\half[\mu]-\frac{\cL^4}{4(\alpha-\beta)^3}\nh{w}^2\right)\nl{\zeta}^2 
   + \left(-\frac{\alpha-\beta}{4}\nh{w}^2\right)\nl{\eta}^2 \leq 0.
  \end{multline}
  Then, adding \eqref{after non all e} and \eqref{after non all z}, we obtain
  \begin{multline}\label{final ineq all}
   \half\dt\nl\eta^2 + \half\dt\nl\zeta^2 
   + \left(\half[\alpha-\beta]-\half[\mu\cI^2 h^2]\right) \left(\nh\eta^2+\nh\zeta^2\right)\\
   + \left[\half[\mu]-\left(\frac{\cL^4}{4(\alpha-\beta)^3}+\frac{\alpha-\beta}{4}\right)\avwf\right]
     \left(\nl{\eta}^2+\nl{\zeta}^2\right) \leq 0.
  \end{multline}
  Thus, defining $Y(t) = \nl{\eta(t)}^2+\nl{\zeta(t)}^2$ and $Z(t) = \nh{v(t)}^2+\nh{w(t)}^2$, we have
  \begin{equation}
   \dt Y + \psi Y \leq 0,
  \end{equation}
  where $\psi(t):=\mu-\left(\frac{\cL^4+(\alpha-\beta)^4}{2(\alpha-\beta)^3}\right)Z(t)$, provided that 
  $\mu\cI^2 h^2\leq\alpha-\beta$.\\
  By Proposition \ref{prop upper bounds} with $T=\frac{1}{\pi^2\nua}$, 
  $\psi$ satisfies \eqref{gronwall sup} and if
  $$ \mu - \frac{\cL^4+(\alpha-\beta)^4}{2T(\alpha-\beta)^3}\cupa G^{2} > 0
     \iff \mu > \frac{\pi^2(\cL^4+(\alpha-\beta)^4)}{\alpha-\beta}G^{2},$$
  then $\psi$ also satisfies \eqref{gronwall inf}, so we can apply Proposition \ref{gronwall}
  to $Y$ and conclude that $(\vt,\wt)$ converges exponentially in time to $(v,w)$.\\
  The requirement on h is 
  $$ h < \left(\frac{\nub}{\cI^2\mu}\right)^{1/2} < \frac{\alpha-\beta}{\pi\cI\sqrt{\cL^4+(\alpha-\beta)^4}}\, G^{-1}, $$
  so $h\sim G^{-1}$.\\
\end{proof}

\subsubsection*[1st Component L2]{
}
\begin{proof}[\bf Proof of Theorem \ref{thm 1st}]~\\
 Let $\eta = v - \vt$ and $\zeta = w - \wt.$ Then $\eta$ satisfies:
 $$ \pt\eta - \alpha\ol\eta - \beta\ol\zeta + \ob{w}{v} - \ob{\wt}{\vt} 
    = -\og(\PT - \tilde{\PT}) - \mu\oI(\eta_1)e_1. $$
 Using the fact that 
 $\ob{w}{v} - \ob{\wt}{\vt} = \ob{\zeta}{v} + \ob{\wt}{\eta}$
 we write:
 $$ \pt\eta - \alpha\ol\eta - \beta\ol\zeta + \ob{\zeta}{v} + \ob{\wt}{\eta} 
     = -\og(\PT - \tilde{\PT}) - \mu\oI(\eta_1)e_1. $$
 Taking the inner product with $\eta$ we obtain:
 $$ \half\dt\nl\eta^2 + \alpha\nh\eta^2 + \beta\ip{\og\zeta}{\og\eta} 
	+ \ib{\zeta}{v}{\eta}
   = -\ip{\og(\PT-\tilde{\PT})}{\eta} - \mu\ip{\oI(\eta_1)}{\eta_1}.
 $$
 Now, by the divergence free condition, we have:
 $$ -\ip{\og(\PT-\tilde{\PT})}{\eta} := -\Ip{\og(\PT-\tilde{\PT})}{\eta} 
    =\Ip{(\PT-\tilde{\PT})}{(\og\cdot\eta)} = 0.
 $$
 By applying Cauchy-Schwarz inequality and \eqref{Young},
 $$
   \a{\beta\ip{\og\zeta}{\og\eta}} 
	\leq \half[\beta]\nh{\eta}^2 + \half[\beta]\nh{\zeta}^2,
 $$
 and by rewriting 
 $\ip{\oI(\eta_1)}{\eta_1} 
	= \ip{\oI(\eta_1)-\eta_1}{\eta_1} + \ip{\eta_1}{\eta_1},$
 we have:
 $$ - \mu\ip{\oI(\eta_1)}{\eta_1} = 
    -\mu\ip{\oI(\eta_1)-\eta_1}{\eta_1} - \mu\nl{\eta_1}^2.
 $$
 Thus, we obtain:
 \begin{align*}
   \half\dt\nl\eta^2 + \left(\alpha-\half[\beta]\right)\nh\eta^2 
		- \half[\beta]\nh\zeta^2 + \mu\nl{\eta_1}^2
    &\leq -\ib{\zeta}{v}{\eta} - \mu\ip{\oI(\eta_1)-\eta_1}{\eta_1}&\\
    &\leq \a{\ib{\zeta}{v}{\eta}} + \mu\a{\ip{\oI(\eta_1)-\eta_1}{\eta_1}}&\\
    &\leq \a{\ib{\zeta}{v}{\eta}} + \mu\nl{\oI(\eta_1)-\eta_1}\nl{\eta_1}&\\
    &\leq \a{\ib{\zeta}{v}{\eta}} + \mu\cI h\nh{\eta_1}\nl{\eta_1}&\\
    &\leq \a{\ib{\zeta}{v}{\eta}} + \half[\mu\cI^2 h^2]\nh{\eta_1}^2 
		+ \half[\mu]\nl{\eta_1}^2,&\\
 \end{align*}
 where in the last three lines we used the Cauchy-Schwarz inequality, the definition of $\oI$, 
and Young's inequality. This leaves us with:
 \begin{equation*}
  \half\dt\nl\eta^2 + \left(\alpha-\half[\beta]\right)\nh\eta^2 
	- \half[\beta]\nh\zeta^2 
  + \half[\mu]\nl{\eta_1}^2 \leq \a{\ib{\zeta}{v}{\eta}} 
	+ \half[\mu\cI^2 h^2]\nh{\eta_1}^2,
 \end{equation*}
 or equivalently,
 \begin{equation}\label{before nonlinear}
  \half\dt\nl\eta^2 + \left(\alpha-\half[\beta]
	-\half[\mu\cI^2 h^2]\right)\nh\eta^2 
  - \half[\beta]\nh\zeta^2 + \half[\mu]\nl{\eta_1}^2 
  \leq \a{\Ib{\zeta}{v}{\eta}}.
 \end{equation}
 Now we apply Lemma \ref{lemma nonlinear term est} to estimate the 
 nonlinear term with  \eqref{nonlin est 1}, yielding:
 \begin{multline}\label{eta eq est}
  \half\dt\nl\eta^2 + \left(\alpha-\half[\beta]-\half[\mu\cI^2 h^2] 
	- c\delta\right)\nh\eta^2 
    +\left(-\half[\beta] - c\delta\right)\nh\zeta^2\\ +
    \left[
      \half[\mu] - \frac{c}{\delta}\nh{v}^2 - \frac{c}{\delta}\nh{v}^2 \ag{\ar{\eta}}
    \right]\nl{\eta_1}^2\\ +
    \left[ 
      -\frac{c}{\delta}\nh{v}^2 - \frac{c}{\delta}\nh{v}^2 \ag{\ar{\zeta}}
    \right]\nl{\zeta_1}^2
    \leq 0.
 \end{multline}
 Proceeding similarly with $\zeta$ we obtain:
 \begin{multline}\label{zeta eq est}
   \half\dt\nl\zeta^2 + \left(\alpha-\half[\beta]-\half[\mu\cI^2 h^2]-c\delta\right)\nh\zeta^2 
     +\left(-\half[\beta]-c\delta\right)\nh\eta^2\\+
     \left[
       \half[\mu] - \frac{c}{\delta}\nh{w}^2 - \frac{c}{\delta}\nh{w}^2 \ag{\ar{\zeta}}
     \right]\nl{\zeta_1}^2 \\+
     \left[ 
       -\frac{c}{\delta}\nh{w}^2 - \frac{c}{\delta}\nh{w}^2 \ag{\ar{\eta}}
     \right]\nl{\eta_1}^2
     \leq 0.
  \end{multline}
Now, adding \eqref{eta eq est} and \eqref{zeta eq est} 
and defining $Z(t)=\nh{v(t)}^2 + \nh{w(t)}^2$,
\begin{multline}\label{comb est}
  \half\dt\nl\eta^2 + \half\dt\nl\zeta^2 + 
   \left(\alpha-\beta-\half[\mu\cI^2 h^2]-2c\delta\right)\nh\eta^2 +
   \left(\alpha-\beta-\half[\mu\cI^2 h^2]-2c\delta\right)\nh\zeta^2\\ +
   \left[
     \half[\mu] - \frac{c}{\delta}\avw - \frac{c}{\delta}\avw \ag{\ar{\eta}}
   \right]\nl{\eta_1}^2\\ +
   \left[ 
     \half[\mu] - \frac{c}{\delta}\avw - \frac{c}{\delta}\avw \ag{\ar{\zeta}}
   \right]\nl{\zeta_1}^2
    \leq 0.
 \end{multline}
 
 Since $\alpha>\beta$, 
 $$\gamma:=(\alpha-\beta)-\half[\mu\cI^2 h^2]-2c\delta \geq \frac{(\alpha-\beta)}{4} > 0,$$
 provided that 
 $h\leq(\alpha-\beta)^{\half}\cI^{-1}\mu^{-\half}$ 
 and by choosing 
 $\delta = \frac{(\alpha-\beta)}{8c}.$\\
 We want to apply Lemma \ref{lemma log min} to the logarithmic terms in \eqref{comb est}.
 To this end note that by \eqref{Poincare}, $\ar{\eta}\geq 1$, so
 $ \ln\ars{\eta} \geq \ln\ar{\eta}.$ 
 Next, we write
 $$ \gamma\nh\eta^2 \geq \half[\gamma]\nh\eta^2 +
  \half[4\pi^2\gamma]\ars{\eta}\nl{\eta_1}^2,$$
 and consider
 \begin{multline*}
   2\pi^2\gamma\ars{\eta}\nl{\eta_1}^2
   - \frac{c}{\delta}\avw \ag{\ars{\eta}}\nl{\eta_1}^2\\
   =2\pi^2\gamma\left(
     \ars{\eta} - \frac{c}{2\pi^2\gamma\delta}\avw\ag{\ars{\eta}}
    \right)\nl{\eta_1}^2.
 \end{multline*}
 By Lemma \ref{lemma log min},
 \begin{equation}\label{log term}
  \ars{\eta} - \frac{c}{2\pi^2\gamma\delta}\avw\ag{\ars{\eta}}
  \geq -\frac{c}{2\pi^2\gamma\delta}\avw\ag[]{\frac{c}{2\pi^2\gamma\delta}\avw}.
 \end{equation}
 Hence, using \eqref{log term} and defining $Y(t) = \nl{\eta(t)}^2 + \nl{\zeta(t)}^2$, 
 we rewrite \eqref{comb est} as
 \begin{equation*}
  \half\dt Y + \half[\gamma]\left(\nh\eta^2 +\nh\zeta^2\right)\\ +
   \left[
     \half[\mu] - \frac{c}{\delta}Z\ag{\frac{c}{2\pi^2\gamma\delta}Z}
   \right]\left( \nl{\eta_1}^2 + \nl{\zeta_1}^2\right) \leq 0.
 \end{equation*}
 By \eqref{Poincare}, 
 $$ \nh\eta^2 +\nh\zeta^2 \geq 4\pi^2\left(\nl\eta^2 +\nl\zeta^2\right)
    \geq 4\pi^2\left(\nl{\eta_2}^2 +\nl{\zeta_2}^2\right), $$
 and so
 \begin{equation}\label{Y est}
  \dt Y + \min\left\{4\pi^2\gamma\,,\,
     \mu - \frac{2c}{\delta}Z\ag{\frac{c}{2\pi^2\gamma\delta}Z}
   \right\}Y \leq 0.
 \end{equation}
 
 Let 
 \begin{equation*}
 \psi(t) := \min\left\{4\pi^2\gamma\,,\,\mu-\frac{2c}{\delta}
                 Z(t)\ag{\frac{c}{2\pi^2\gamma\delta}Z(t)}\right\},
 \end{equation*}
 and in order to apply Proposition \ref{gronwall} we only need to show that $\psi$ satisfies
 \eqref{gronwall inf} and \eqref{gronwall sup}. 
 It is sufficient to show that for some $T,t_0>0$,
 \begin{equation} \label{gronwall inf 1}
  \mu - \limsup_{t\to\infty}\frac{1}{T}\int_{t}^{t+T}
  \frac{2c}{\delta} Z(s)\ag{\frac{c}{2\pi^2\gamma\delta}Z(s)} \,ds >0,
 \end{equation}
 and
 \begin{equation}\label{gronwall sup 1}
  \sup_{s>t_0}Z(s)\ag{\frac{c}{2\pi^2\gamma\delta}Z(s)} \,ds<\infty.
 \end{equation}
In fact, \eqref{gronwall sup 1} follows directly from \eqref{sup bound} with the $t_0$ given there.

 To see \eqref{gronwall inf 1}, by Proposition \ref{prop upper bounds} with 
 $T=\frac{1}{\pi^2\nua}$, we have:
 \begin{multline*}
 \lim_{t\to\infty}\frac{1}{T}\int_{t}^{t+T}
   \frac{2c}{\delta} Z(s)\ag{\frac{c}{2\pi^2\gamma\delta}Z(s)} \,ds\\
 \leq \frac{2c}{\delta T}\ag{\frac{c}{2\pi^2\gamma\delta}\cupb G^2e^{\C G^4}}
   \lim_{t\to\infty}\int_{t}^{t+T}Z(s) \,ds\\
 \leq \frac{2c}{\delta T}\left(\tilde{c}+2\ln G + \C G^4\right)\cupa G^2,\\
 =32\pi^2c^2\nua\left(\tilde{c}+2\ln G + \C G^4\right)G^2.\\
 \end{multline*}
 Therefore, \eqref{gronwall inf 1} holds by choosing 
 $\mu > 32\pi^2c^2\nua\left(\tilde{c}+2\ln G + \C G^4\right)G^2$.
 In addition, the requirement $h\leq\frac{(\nub)^{\half}}{\cI}\mu^{-\half}$ implies 
 $h\sim G^{-3}$.\\
\end{proof}

\subsubsection*[v only L2]{
}
\begin{proof}[\bf Proof of Theorem \ref{thm v}]~\\
 Let $\eta = v - \vt$ and $\zeta = w - \wt.$
 Similarly to how we showed \eqref{before nonlinear}, the equation we obtain for $\eta$ is
 \begin{equation}\label{before non v}
  \half\dt\nl\eta^2 + \left(\alpha-\half[\beta]-\half[\mu\cI^2 h^2]\right)\nh\eta^2 
  - \half[\beta]\nh\zeta^2 + \half[\mu]\nl{\eta}^2 \leq \a{\Ib{\zeta}{v}{\eta}},
 \end{equation}
 but now the equation for $\zeta$ is
 \begin{equation}\label{before non v.2}
  \half\dt\nl\zeta^2 + \left(\alpha-\half[\beta]\right)\nh\zeta^2 
  - \half[\beta]\nh\eta^2 \leq \a{\Ib{\eta}{w}{\zeta}}.
 \end{equation}
 
 We estimate the integral in \eqref{before non v} using \eqref{nonlin est 2}, so 
 \eqref{before non v} becomes:
 \begin{multline}\label{after non 1}
  \half\dt\nl\eta^2 + \left(\alpha-\half[\beta]-\half[\mu\cI^2 h^2]-\half[\epsilon]\right)\nh\eta^2 
  +\left(-\half[\beta]-\frac{\cL\delta}{4}\right)\nh\zeta^2 \\
  + \left(\half[\mu]-\frac{\cL^2}{8\epsilon\delta^2}\nh{v}^2\right)\nl{\eta}^2 
  + \left(-\frac{\cL\delta}{4}\nh{v}^2\right)\nl{\zeta}^2 \leq 0,
 \end{multline}
 Similarly, we estimate the integral in \eqref{before non v.2} using \eqref{nonlin est 2.1}, and get:
 \begin{multline}\label{after non 1.2}
  \half\dt\nl\zeta^2 
  +\left(\alpha-\half[\beta]-\frac{\cL\delta}{4}\right)\nh\zeta^2 
  +\left(-\half[\beta]-\half[\epsilon]\right)\nh\eta^2 \\
  +\left(-\frac{\cL\delta}{4}\nh{w}^2\right)\nl{\zeta}^2
  +\left(-\frac{\cL^2}{8\epsilon\delta^2}\nh{w}^2\right)\nl{\eta}^2\leq 0.
 \end{multline}
 Adding \eqref{after non 1} and \eqref{after non 1.2},
 \begin{multline*}
   \half\dt\nl\eta^2 + \half\dt\nl\zeta^2 + 
   \left(\alpha-\beta-\half[\mu\cI^2 h^2]-\epsilon\right)\nh\eta^2 
   +\left(\alpha-\beta-\half[\cL\delta]\right)\nh\zeta^2 \\
   + \left(\half[\mu]-\frac{\cL^2}{8\epsilon\delta^2}\avwf\right)\nl{\eta}^2 
   + \left(-\frac{\cL\delta}{4}\avwf\right)\nl{\zeta}^2 \leq 0.
 \end{multline*}
 Now, if we choose 
 $$h \leq \frac{(\alpha - \beta)^{1/2}}{\cI}\mu^{-1/2},$$ 
 and $\epsilon=\half[\alpha-\beta]$, then 
 $\alpha-\beta-\half[\mu\cI^2 h^2]-\epsilon \geq 0$.\\ 
 Also, by choosing $\delta < \frac{\alpha-\beta}{\cL}$, we have 
 $$\gamma:=\alpha-\beta-\half[\cL\delta]>\half[\alpha-\beta]>0.$$
 Then by applying \eqref{Poincare} we obtain 
 $ \gamma \nh{\zeta}^2 \geq \gamma 4\pi^2 \nl \zeta ^2.$ 
 Hence, defining $Y(t) = \nl{\eta(t)}^2+\nl{\zeta(t)}^2$ and $Z(t) = \nh{v(t)}^2+\nh{w(t)}^2$, we have:
 \begin{equation}
  \dt Y + \psi Y \leq 0,
 \end{equation}
 where $\psi(t):=\min\left\{\mu-\frac{\cL^2}{4\epsilon\delta^2}Z(t)\,
    			,\,8\pi^2\gamma-\frac{\cL\delta}{2}Z(t)\right\}.$
 Using Proposition \ref{prop upper bounds} similarly as before, with $T=\frac{1}{\pi^2\nua}$,
 $\psi$ satisfies \eqref{gronwall sup} as well as \eqref{gronwall inf}
 provided that 
  $$ \delta < \frac{\nub}{\cL}\,\frac{4}{4+\nua^2G^2}
     \implies 8\pi^2\gamma-\frac{\cL\delta}{2T}\cupa G^2 > 4\nua>0,$$
 and
  $$\mu > \frac{\pi^2\cL^4 G^2(4+\nua^2 G^2)^2}{16\nua}\implies
    \mu-\frac{\cL^2}{4\epsilon\delta^2T}\cupa G^2 > 0.$$
 By choosing such a $\mu$ and $\delta$, we can apply Proposition \ref{gronwall} to conclude that
 $(\vt,\wt)$ converges exponentially in time to $(v,w)$.\\
 Now the requirement we needed on $h$ implies
 $$h < \frac{4\nua}{\pi\cI\cL^2 G(4+\nua^2 G^2)},$$
 so $h\sim G^{-3}.$\\
 \end{proof}

 \subsection{Proof of $H^1$ Convergence Results with Type 1 Interpolants}

 \begin{proof}[\bf{Proof of Theorem~\ref{thm H1 all}.}]\quad\\
 By denoting $\eta = v - \vt$ and $\zeta = w - \wt$
 and subtracting the equations for $(v, w)$ and $(\vt, \wt)$,  
 we obtain the following equation for $\eta$ and $\zeta$
 \begin{align*}
      &
      \pt\eta - \alpha\ol\eta - \beta\ol\zeta + \ob{\zeta}{v} + \ob{\wt}{\eta} 
      =
      -\og(\PT - \tilde{\PT}) - \mu\oI(\eta), 
      \\&
      \pt\zeta - \alpha\ol\zeta - \beta\ol\eta + \ob{\eta}{w} + \ob{\vt}{\zeta} 
      =
      -\og(\PT - \tilde{\PT}) - \mu\oI(\zeta).
 \end{align*}
 Taking the inner product with $-\Delta\eta$ and $-\Delta\zeta$, respectively, we obtain:
 \begin{align*}
      &
      \frac{1}{2}\frac{d}{d t}\nh{\eta}^2
      +
      \alpha\nH{\eta}^2
      =
      -\beta\ip{\Delta\zeta}{\Delta\eta}
      +
      \ip{\ob{\zeta}{v}}{\Delta\eta}
      +
      \ip{\ob{\wt}{\eta}}{\Delta\eta}
      +
      \ip{\og(\PT - \tilde{\PT})}{\Delta\eta}
      +
      \mu\ip{\oI(\eta)}{\Delta\eta},
      \\&
      \frac{1}{2}\frac{d}{d t}\nh{\zeta}^2
      +
      \alpha\nH{\zeta}^2
      =
      -\beta\ip{\Delta\eta}{\Delta\zeta}
      +
      \ip{\ob{\eta}{w}}{\Delta\zeta}
      +
      \ip{\ob{\vt}{\zeta}}{\Delta\zeta}
      +
      \ip{\og(\PT - \tilde{\PT})}{\Delta\zeta}
      +
      \mu\ip{\oI(\zeta)}{\Delta\zeta}.
 \end{align*}
 Then, by the divergence-free condition,
 $$ \ip{\og(\PT-\tilde{\PT})}{\ol\eta} = -\Ip{(\PT-\tilde{\PT})}{\ol(\og\cdot\eta)} = 0,
 $$
 and similarly
 $$ \ip{\og(\PT-\tilde{\PT})}{\ol\zeta} = 0.
 $$
 Also, by applying Cauchy-Schwarz inequality and \eqref{Young}, we have
 $$
 {-\beta\ip{\ol\zeta}{\ol\eta}} \leq \half[\beta]\nH{\eta}^2 + \half[\beta]\nH{\zeta}^2.
 $$
 Rewriting 
 $\ip{\oI(\eta)}{-\Delta\eta} = \ip{\oI(\eta)-\eta}{-\Delta\eta} + \ip{\eta}{\Delta\eta},$
 we have,
 $$ - \mu\ip{\oI(\eta)}{-\ol \eta} = 
    -\mu\ip{\oI(\eta)-\eta}{-\Delta\eta} - \mu\nh{\eta}^2, 
 $$
 and similarly, 
 $$ - \mu\ip{\oI(\zeta)}{-\ol \zeta} = 
    -\mu\ip{\oI(\zeta)-\zeta}{-\Delta\zeta} - \mu\nh{\zeta}^2.
 $$
 Adding up the equations for $\eta$ and $\zeta$, we obtain 
 \begin{align*}
      &
      \frac{1}{2}\frac{d}{d t}\left( \nh{\eta}^2 + \nh{\zeta}^2 \right)
      +
      (\alpha - \beta)\left( \nH{\eta}^2 + \nH{\zeta}^2 \right)
      \\&
      \leq
      \a{\ip{\ob{\zeta}{v}}{\Delta\eta}}
      +
      \a{\ip{\ob{\eta}{w}}{\Delta\zeta}}
      +
      \a{\ip{\ob{\wt}{\eta}}{\Delta\eta}}
      +
      \a{\ip{\ob{\vt}{\zeta}}{\Delta\zeta}}
      \\&\qquad
      +
      \mu\a{\ip{\oI(\eta)-\eta}{\Delta\eta}}
      +
      \mu\a{\ip{\oI(\zeta)-\zeta}{\Delta\zeta}}
      -
      \mu\left(\nh{\eta}^2 + \nh{\zeta}^2\right).
 \end{align*}
 Due to the properties of $\oI$, we have
 \begin{align*}
     \mu\a{\ip{\oI(\eta)-\eta}{\ol\eta}} 
     &
     \leq
     \mu\nl{\oI(\eta)-\eta}\nH{\eta}
     \leq 
     \mu\cI h \nh{\eta}\nH{\eta} 
     \\&
     \leq 
     \frac{4\mu^2\cI^2 h^2}{\alpha-\beta}\nh{\eta}^2 
     + 
     \frac{\alpha-\beta}{16}\nH{\eta}^2,
 \end{align*}
 and similarly, we obtain 
 \begin{align*}
     \mu\a{\ip{\oI(\zeta)-\zeta}{\ol\zeta}} 
     &
     \leq
     \frac{4\mu^2\cI^2 h^2}{\alpha-\beta}\nh{\zeta}^2 
     + 
     \frac{\alpha-\beta}{16}\nH{\zeta}^2. 
 \end{align*}
 Next, we estimate the nonlinear terms. 
 First, by H\"older's and Sobolev inequalities, we obtain 
 \begin{align*}
      \a{\ip{\ob{\zeta}{v}}{\Delta\eta}}
      &
      \leq
      \I{|\zeta| |\nabla v| |\Delta\eta|}
      \leq
      \nL{\zeta} \Vert\nabla v\Vert_{L^4} \nH{\eta}
      \\&
      \leq
      \nl{\zeta}^{1/2} \nh{\zeta}^{1/2} \nh{v}^{1/2} \nH{v}^{1/2} \nH{\eta}
      \\&
      \leq
      \frac{4}{\alpha - \beta}\nh{v} \nH{v} \nl{\zeta} \nh{\zeta}
      +
      \frac{\alpha-\beta}{16}\nH{\eta}^2
      \\&
      \leq
      \frac{4}{2\pi(\alpha - \beta)}\nh{v} \nH{v} \nl{\zeta} \nH{\zeta}
      +
      \frac{\alpha-\beta}{16}\nH{\eta}^2
      \\&
      \leq
      \frac{1}{4\pi^2}\left(\frac{4}{\alpha - \beta}\right)^3 \nh{v}^2 \nH{v}^2 \nl{\zeta}^2
      +
      \frac{\alpha-\beta}{16}\left(\nH{\eta}^2 + \nH{\zeta}^2\right),
 \end{align*}
 where we used Poincar\'e's and Young's inequalities. 
 The estimate for $\ip{\ob{\eta}{w}}{\Delta\zeta}$ is similarly, i.e., we have 
 \begin{align*}
      \a{\ip{\ob{\eta}{w}}{\Delta\zeta}}
      &
      \leq
      \frac{1}{4\pi^2}\left(\frac{4}{\alpha - \beta}\right)^3 \nh{w}^2 \nH{w}^2 \nl{\eta}^2
      +
      \frac{\alpha-\beta}{16}\left(\nH{\eta}^2 + \nH{\zeta}^2\right). 
 \end{align*} 
 Regarding $\ip{\ob{\wt}{\eta}}{\Delta\eta}$, 
 we first rewrite it as 
 $$ \ip{\ob{\wt}{\eta}}{\Delta\eta} = \ip{\ob{w}{\eta}}{\Delta\eta} - \ip{\ob{\zeta}{\eta}}{\Delta\eta} = {\uppercase\expandafter{\romannumeral1}} + {\uppercase\expandafter{\romannumeral2}}. 
 $$ 
 In order to estimate ${\uppercase\expandafter{\romannumeral1}}$, 
 we first observe that by the periodic boundary conditions, we have
 \begin{align}\label{fact 1}
      \nh{\eta}^2
      &
      =
      \I{\nabla\eta\cdot\nabla\eta}
      =
      - \I{\eta\Delta\eta}
      \leq
      \nl{\eta}\nH{\eta}.
 \end{align}
 Thus, we integrate by parts and proceed to estimate ${\uppercase\expandafter{\romannumeral1}}$ as
 \begin{align*}
      \ip{\ob{w}{\eta}}{\Delta\eta}
      &
      =
      \sum_{i, j, k=1}^{2}\I{w_{i}\partial_{i}\eta_{k}\partial_{jj}^2\eta_{k}}
      =
      -\sum_{i, j, k=1}^{2}\I{\partial_{j}w_{i}\partial_{i}\eta_{k}\partial_{j}\eta_{k}}
      \\&
      \leq
      \I{|\nabla w| |\nabla\eta|^2}
      \leq
      \Vert\nabla w\Vert_{L^2} \Vert\nabla\eta\Vert_{L^2} \Vert\Delta\eta\Vert_{L^2}
      \\&
      \leq
      \frac{4}{\alpha - \beta} \Vert\nabla w\Vert_{L^2}^2 \nh{\eta}^2
      +
      \frac{\alpha-\beta}{16}\nH{\eta}^2
      \\&
      \leq
      \frac{4}{\alpha - \beta} \Vert\nabla w\Vert_{L^2}^2 \nl{\eta} \nH{\eta}
      +
      \frac{\alpha-\beta}{16}\nH{\eta}^2
      \\&
      \leq
      \left(\frac{4}{\alpha - \beta}\right)^3 \nh{w}^4 \nl{\eta}^2
      +
      \frac{\alpha-\beta}{8}\nH{\eta}^2. 
 \end {align*}
By similar estimates and the analogy of \eqref{fact 1} for $\zeta$, i.e.,
 \begin{equation*}
      \nh{\zeta}^2
      \leq
      \nl{\zeta}\nH{\zeta}, 
 \end{equation*}
 we estimate ${\uppercase\expandafter{\romannumeral2}}$ as 
 \begin{align*}
      - \ip{\ob{\zeta}{\eta}}{\Delta\eta}
      &
      \leq
      \I{|\nabla\zeta| |\nabla\eta|^2}
      \leq
      \nh{\zeta} \nh{\eta} \nH{\eta}
      \\&
      \leq
      \frac{4}{\alpha - \beta} \nh{\eta}^2 \nh{\zeta}^2
      +
      \frac{\alpha-\beta}{16}\nH{\eta}^2
      \\&
      \leq
      \frac{4}{\alpha - \beta} \nl{\eta}\nl{\zeta}\nH{\eta}\nH{\zeta}
      +
      \frac{\alpha-\beta}{16}\nH{\eta}^2
      \\&
      \leq
      \frac{2}{\alpha - \beta} \nl{\eta}\nl{\zeta}\left(\nH{\eta}^2 + \nH{\zeta}^2\right)
      +
      \frac{\alpha-\beta}{16}\nH{\eta}^2.
 \end{align*}
 By a similar approach, we have
 $$ \ip{\ob{\vt}{\zeta}}{\Delta\zeta} = \ip{\ob{v}{\zeta}}{\Delta\zeta} - \ip{\ob{\eta}{\zeta}}{\Delta\zeta} = {\uppercase\expandafter{\romannumeral3}} + {\uppercase\expandafter{\romannumeral4}}, 
 $$
 and ${\uppercase\expandafter{\romannumeral3}}$ is bounded by 
 \begin{align*}
      \a{\ip{\ob{v}{\zeta}}{\Delta\zeta}}
      &
      \leq
      \left(\frac{4}{\alpha - \beta}\right)^3 \nh{v}^4 \nl{\zeta}^2
      +
      \frac{\alpha-\beta}{8}\nH{\zeta}^2,
 \end{align*}
 while we estimate ${\uppercase\expandafter{\romannumeral4}}$ as
 \begin{align*}
      - \ip{\ob{\eta}{\zeta}}{\Delta\zeta}
      &
      \leq
      \frac{2}{\alpha - \beta} \nl{\eta}\nl{\zeta}\left(\nH{\zeta}^2 + \nH{\eta}^2\right)
      +
      \frac{\alpha-\beta}{16}\nH{\zeta}^2.
 \end{align*}
 Combining all the above estimates, we obtain 
 \begin{align}\label{h1 final ineq all alt}
      &
      \frac{1}{2}\frac{d}{d t}\left( \nh{\eta}^2 + \nh{\zeta}^2 \right)
      +
      \underbrace{\left( \frac{\alpha - \beta}{2} - \frac{4}{\alpha - \beta} \nl{\eta}\nl{\zeta}\right)}_{{\uppercase\expandafter{\romannumeral5}}} \left( \nH{\eta}^2 + \nH{\zeta}^2 \right) \nonumber
      \\&
      \leq
      \bigg[\underbrace{\frac{1}{4\pi^2}\left(\frac{4}{\alpha - \beta}\right)^3 \left(\nh{v}^2 \nH{v}^2 + \nh{w}^2 \nH{w}^2\right)}_{{\uppercase\expandafter{\romannumeral6}}} \nonumber
      \\&\qquad
      +
      \underbrace{\left(\frac{4}{\alpha - \beta}\right)^3 \left(\nh{v}^4 + \nh{w}^4\right)}_{{\uppercase\expandafter{\romannumeral7}}} \bigg] \left( \nl{\eta}^2 + \nl{\zeta}^2 \right) \nonumber
      \\&\qquad
      +
      \bigg(\underbrace{\frac{4\mu^2\cI^2 h^2}{\alpha-\beta}}_{{\uppercase\expandafter{\romannumeral8}}} - \mu\bigg)\left( \nh{\eta}^2 + \nh{\zeta}^2 \right).
 \end{align}
 Now choose $h$ such that
 \begin{align*}
      {\uppercase\expandafter{\romannumeral8}}
      &
      =
      \frac{4\mu^2\cI^2 h^2}{\alpha-\beta}
      <
      \frac{\mu}{2}.
 \end{align*}
 Thus, we have 
 \begin{align}\label{h mu relation}
      h^2 < \frac{\alpha-\beta}{8\mu\cI^2}.
 \end{align}
%
%
Moreover, by Theorem~\ref{thm all}, 
we know that after a sufficiently large time $T_1$, $\nl{\eta}$ and $\nl{\zeta}$ are small enough. 
so that we have 
\begin{align}\label{H1_via_L2_One}
\nl{\eta}\nl{\zeta} \leq \frac{(\alpha - \beta)^2}{16}, 
\end{align}
 which implies that
 $ \uppercase\expandafter{\romannumeral5} \geq 0 $, 
so we have:
$$
      \frac{1}{2}\frac{d}{d t}\left( \nh{\eta}^2 + \nh{\zeta}^2 \right)
      + \half[\mu]\left( \nh{\eta}^2 + \nh{\zeta}^2 \right)
      \leq
      \left({\uppercase\expandafter{\romannumeral6}+\uppercase\expandafter{\romannumeral7}}\right)\left( \nl{\eta}^2 + \nl{\zeta}^2 \right).
$$
Define $Y(t))= \nh{\eta(t)}^2 + \nh{\zeta(t)}^2$, and by appealing to Proposition~\ref{prop upper bounds},
we see that ${\uppercase\expandafter{\romannumeral6}+\uppercase\expandafter{\romannumeral7}}$ is bounded by some number $\half[M_G]$. Also, by Theorem~\ref{thm all} we know that there exists constants $K,a>0$ such that
$\nl{\eta(t)}^2 + \nl{\zeta(t)}^2 \leq Ke^{-a t},\, \forall t\geq T_1$.
Putting all of this together, we have the following for all $t>T_1$:
\begin{align*}
  \dt Y(t) + \mu Y(t) &\leq M_GKe^{-a t},\\
 \Rightarrow \dt \left( e^{\mu t}Y(t) \right) &\leq M_GKe^{(\mu-a)t},\\
  \Rightarrow e^{\mu t}Y(t)-e^{\mu T_1}Y(T_1) &\leq \frac{M_GK}{\mu-a}e^{(\mu-a)t} - \frac{M_GK}{\mu-a}e^{(\mu-a)T_1},\\
  \Rightarrow Y(t)& \leq Y(T_1)e^{-\mu (t-T_1)} + \frac{M_GK}{\mu-a}\left(e^{-at} - e^{-\mu(t-T_1)-aT_1}\right).
\end{align*}
Therefore, $Y(t))= \nh{\eta(t)}^2 + \nh{\zeta(t)}^2 \to 0$ exponentially as $t\to\infty$ as long as $\mu$ and $h$ satisfy the conditions of Theorem~\ref{thm all}, as well as the new requirement \eqref{h mu relation}. So, choosing
$$ \mu > \frac{\pi^2(\cL^4+(\alpha-\beta)^4)}{\alpha-\beta}G^{2},
\quad\text{and}\quad 
h < \frac{\alpha-\beta}{2\sqrt{2}\pi\cI\sqrt{\cL^4+(\alpha-\beta)^4}}\, G^{-1},
$$
we have exponential convergence.
%
%
%
%
 \end{proof}
 
 Next, we prove the $H^1$ decay estimates for the data assimilation scenario where measurement is only on $v_1$ and $w_1$.
 \begin{proof}[\bf{Proof of Theorem~\ref{thm H1 1st}.}]~\\ 
 We still denote the difference of solutions to \eqref{MHD vw} and \eqref{MHD data 1 comp} by $\eta = v - \vt$ and $\zeta = w - \wt$. 
 Similarly to the beginning of the proof of Theorem~\ref{thm H1 all}, we have 
 \begin{align*}
      &
      \frac{1}{2}\frac{d}{d t}\left( \nh{\eta}^2 + \nh{\zeta}^2 \right)
      +
      (\alpha - \beta)\left(\nH{\eta}^2 + \nH{\zeta}^2\right)
      \\&
      \leq
      \a{\ip{\ob{\zeta}{v}}{\Delta\eta}}
      +
      \a{\ip{\ob{\eta}{w}}{\Delta\zeta}}
      +
      \a{\ip{\ob{\wt}{\eta}}{\Delta\eta}}
      +
      \a{\ip{\ob{\vt}{\zeta}}{\Delta\zeta}}
      \\&\qquad
      +
      \mu\a{\ip{\oI(\eta_1)-\eta_1}{\Delta\eta_1}}
      +
      \mu\a{\ip{\oI(\zeta_1)-\zeta_1}{\Delta\zeta_1}}
      -
      \mu\nh{\eta_1}^2
      -
      \mu\nh{\zeta_1}^2,
 \end{align*}
 as well as 
 \begin{align*}
     \mu\a{\ip{\oI(\eta_1)-\eta_1}{\ol\eta_1}} 
     &
     \leq
     \frac{4\mu^2\cI^2 h^2}{\alpha-\beta}\nh{\eta_1}^2
     + 
     \frac{\alpha-\beta}{16}\nH{\eta_1}^2, 
 \end{align*}
 and 
 \begin{align*}
     \mu\a{\ip{\oI(\zeta_1)-\zeta_1}{\ol\zeta_1}} 
     &
     \leq
     \frac{4\mu^2\cI^2 h^2}{\alpha-\beta}\nh{\zeta_1}^2
     + 
     \frac{\alpha-\beta}{16}\nH{\zeta_1}^2. 
 \end{align*}
 The estimates for the nonlinear terms are also similar. 
 Namely, we have
 \begin{align*}
      \a{\ip{\ob{\zeta}{v}}{\Delta\eta}}
      &
      \leq
      \frac{1}{4\pi^2}\left(\frac{4}{\alpha - \beta}\right)^3 \nh{v}^2 \nH{v}^2 \nl{\zeta}^2
      +
      \frac{\alpha-\beta}{16}\left(\nH{\eta}^2 + \nH{\zeta}^2\right),
 \end{align*}
 and
 \begin{align*}
      \a{\ip{\ob{\eta}{w}}{\Delta\zeta}}
      &
      \leq
      \frac{1}{4\pi^2}\left(\frac{4}{\alpha - \beta}\right)^3 \nh{w}^2 \nH{w}^2 \nl{\eta}^2
      +
      \frac{\alpha-\beta}{16}\left(\nH{\eta}^2 + \nH{\zeta}^2\right).  
 \end{align*}
 Also, by rewriting 
 $$ \ip{\ob{\wt}{\eta}}{\Delta\eta} = \ip{\ob{w}{\eta}}{\Delta\eta} - \ip{\ob{\zeta}{\eta}}{\Delta\eta}   
 $$ 
 we obtain 
 \begin{align*}
      \ip{\ob{w}{\eta}}{\Delta\eta}
      &
      \leq
      \left(\frac{4}{\alpha - \beta}\right)^3 \nh{w}^4 \nl{\eta}^2
      +
      \frac{\alpha-\beta}{8}\nH{\eta}^2, 
 \end {align*}
 and 
 \begin{align*}
      - \ip{\ob{\zeta}{\eta}}{\Delta\eta}
      &
      \leq
      \frac{2}{\alpha - \beta} \nl{\eta}\nl{\zeta}\left(\nH{\eta}^2 + \nH{\zeta}^2\right)
      +
      \frac{\alpha-\beta}{16}\nH{\eta}^2.
 \end{align*}

 Estimates for 
 $$ \ip{\ob{\vt}{\zeta}}{\Delta\zeta} = \ip{\ob{v}{\zeta}}{\Delta\zeta} - \ip{\ob{\eta}{\zeta}}{\Delta\zeta}   
 $$
 also follow similarly, and we obtain 
 \begin{align*}
      \ip{\ob{v}{\zeta}}{\Delta\zeta}
      &
      \leq
      \left(\frac{4}{\alpha - \beta}\right)^3 \nh{v}^4 \nl{\zeta}^2
      +
      \frac{\alpha-\beta}{8}\nH{\zeta}^2,
 \end{align*}
 and 
 \begin{align*}
      - \ip{\ob{\eta}{\zeta}}{\Delta\zeta}
      &
      \leq
      \frac{2}{\alpha - \beta} \nl{\eta}\nl{\zeta}\left(\nH{\zeta}^2 + \nH{\eta}^2\right)
      +
      \frac{\alpha-\beta}{16}\nH{\zeta}^2.
 \end{align*}
 
 Combining all the above estimates, we obtain 
 \begin{align}\label{h1 final ineq 1 comp}
      &
      \frac{1}{2}\frac{d}{d t}\left( \nh{\eta}^2 + \nh{\zeta}^2 \right)
      +
      \underbrace{\left( \frac{\alpha - \beta}{2} - \frac{4}{\alpha - \beta} \nl{\eta}\nl{\zeta}\right)}_{{\uppercase\expandafter{\romannumeral5}}}\left( \nH{\eta}^2 + \nH{\zeta}^2 \right) \nonumber
      \\&\qquad
      \leq
      \bigg[\underbrace{\left(\frac{1}{4\pi^2}\left(\frac{4}{\alpha - \beta}\right)^3 \left(\nh{v}^2 \nH{v}^2 + \nh{w}^2 \nH{w}^2\right)\right)}_{{\uppercase\expandafter{\romannumeral6}}} \nonumber
      \\&\qquad\qquad
      +
      \underbrace{\left(\frac{4}{\alpha - \beta}\right)^3 \left(\nh{v}^4 + \nh{w}^4\right)}_{{\uppercase\expandafter{\romannumeral7}}} \bigg] \left( \nl{\eta}^2 + \nl{\zeta}^2 \right) \nonumber
      \\&\qquad\qquad
      +
      \bigg(\underbrace{\frac{4\mu^2\cI^2 h^2}{\alpha-\beta}}_{{\uppercase\expandafter{\romannumeral8}}} - \mu\bigg)\left( \nh{\eta_1}^2 + \nh{\zeta_1}^2 \right).
 \end{align}
We choose $h$ such that 
\begin{align}\label{comp_1_mu_choice}
{\uppercase\expandafter{\romannumeral8}}  = \frac{4\mu^2\cI^2 h^2}{\alpha-\beta} < \frac{\mu}{2}.
\end{align}
In view of Theorem~\ref{thm 1st}, 
after sufficiently large time $T_2>0$, $\nl{\eta}$ and $\nl{\zeta}$ are small enough so that 
\begin{align}\label{H1_via_L2_Two}
\nl{\eta} \nl{\zeta} < \frac{(\alpha - \beta)^2}{16}.
\end{align}
Thus, ${\uppercase\expandafter{\romannumeral5}}>\frac{1}{4}(\alpha-\beta)>0$.  Let us denote $Y(t) = \nh{\eta}^2 + \nh{\zeta}^2$.   
 Then, for all $t>T_2$, by applying Poincar\'e's inequality to the second term on the left-hand side of \eqref{h1 final ineq 1 comp}, it follows, due to \eqref{comp_1_mu_choice}, that
 \begin{align*}
      \frac{1}{2}\frac{d}{d t}Y(t)
      +
      \pi^2\nua Y(t) 
      &\leq
      M_G
      \left( \nl{\eta}^2 + \nl{\zeta}^2 \right) 
      +
      \left({\uppercase\expandafter{\romannumeral8}} - \mu\right)\left( \nh{\eta_1}^2 + \nh{\zeta_1}^2 \right)
      \\&\leq
      M_G
      \left( \nl{\eta}^2 + \nl{\zeta}^2 \right)
      \\&\leq
      K' M_G e^{-a't},
 \end{align*}
 where $K'>0$ and $a'>0$ chosen so that is such that $\nl{\eta}^2 + \nl{\zeta}^2\leq K' M_G e^{-a't}$ for all $t>T_2$ (this is permitted due to Theorem \ref{thm 1st}).
 This implies
 \begin{align*}
\frac{d}{d t}\left(Y(t)e^{2\pi^2\nua t}\right)
\leq
      K'M_G e^{2\pi^2\nua t} e^{-a't}.
 \end{align*}
Integrating, we arrive at
  \begin{align*}
Y(t)
\leq 
Y(T_2) e^{-2\pi^2\nua(t-T_2)}
+ \frac{K'M_G}{2\pi^2\nua - a'}\left(e^{-ta'} - e^{-2\pi^2\nua(t-T_2) -a'T_2}\right).
 \end{align*}
 (Note that, if necessary, one may choose $a'$ slightly smaller so that $2\pi^2\nua \neq a'$.)  In particular, $Y(t)= \nh{\eta}^2 + \nh{\zeta}^2$ decays exponentially in time for all $t> T_2$, with $h$  and $\mu$ chosen so that
$$
\mu > 32\pi^2c^2\nua\left(\tilde{c}+2\ln G + \C G^4\right)G^2
$$
and 
$$
h < (2\sqrt{2}\cI)^{-1}\nua^{\half}\mu^{-\half} < (8\sqrt{2}\pi\cI c)^{-1}\left(\tilde{c}+2\ln G + \C G^4\right)^{-\half}G^{-1}.
$$
Thus, the proof of Theorem~\ref{thm H1 1st} is complete.
\end{proof}

\begin{proof}[\bf{Proof of Theorem~\ref{thm H1 v}.}]~\\ 
The proof goes similarly as that of Theorem~\ref{thm H1 1st}. For the sake of simplicity, we omit the details here. 
\end{proof}
 
 \subsection{Proofs of the Results for Type 2 Interpolants}
 
 \begin{lemma}\label{lemma nonlinear term est 2}
 Let $u,v,w\in \sSob{2}$ be divergence free. Then the following inequalities hold:
 \begin{align}\label{nonlin est 1 T2}
  (a)\quad\a{ \Ib u v {\ol w} } \leq 3\cT\nh{u_1}\nh{v}\nH{w}\lnTwo{u}\nonumber\\
   		+(\cT+4\cB)\nH{u}\nh{v}\nh{w_1}\lnTwo{w}\nonumber\\
  	 	+ 2\cT\nh{u}\nH{v}\nh{w_1}\lnTwo{w},
 \end{align}
  \begin{equation}\label{nonlin est 2 T2}
  (b)\quad\a{ \Ib u v {\ol v} } \leq (2\cB+5\cT)\nh{u}\nh{v_1}\nH{v}\lnTwo{v}.
  \end{equation}
 \end{lemma}
 \begin{proof}
 See the appendix.
 \end{proof}
 
 \subsubsection*[Proof of Type 2 w/ 1st comp]{
 }
 In the following proof of Theorem~\ref{T2 thm 1}, we simultaneously establish a bound like \eqref{sup bound} for the data assimilation solution, because the proof requires such an estimate. 
 \begin{proof}[{\bf Proof of Theorem~\ref{T2 thm 1}}]~\\
 Since $(\vt,\wt)$ is a strong solution and $\vt_0\equiv\wt_0\equiv0$, there is a largest time $T_0\in(0,\infty]$ such that
 $$
 \sup_{t\in[0,T_0)}(\nh{\vt(t)}^2+\nh{\wt(t)}^2) \leq \cupb[50] G^2 e^{\C G^4}.
 $$
 Suppose that $T_0<\infty$.\\
 Then we know that
 \begin{equation}\label{fin exist time sup}
 \limsup_{t\to T_0^-}\:(\nh{\vt(t)}^2+\nh{\wt(t)}^2)
 	= \sup_{t\in[0,T_0)}(\nh{\vt(t)}^2+\nh{\wt(t)}^2) = \cupb[50] G^2 e^{\C G^4}.
 \end{equation}
 Let $\eta = v - \vt$ and $\zeta = w - \wt.$ Then we have the following equation for $\eta$:
 $$ \pt\eta - \alpha\ol\eta - \beta\ol\zeta + \ob{\zeta}{v} + \ob{\wt}{\eta} 
     = -\og(\PT - \tilde{\PT}) - \mu\oI(\eta_1)e_1. $$
 Taking the inner product with $-\ol\eta$, we obtain:
 $$ \half\dt\nh\eta^2 + \alpha\nH\eta^2 + \beta\ip{\ol\zeta}{\ol\eta} 
 	- \ib{\zeta}{v}{\ol\eta} - \ib{\wt}{\eta}{\ol\eta}
    = \ip{\og(\PT-\tilde{\PT})}{\ol\eta} - \mu\ip{\oI(\eta_1)}{-\ol\eta_1}
 $$
 Now, by the divergence free condition, we have:
 $$ \ip{\og(\PT-\tilde{\PT})}{\ol\eta} = -\Ip{(\PT-\tilde{\PT})}{\ol(\og\cdot\eta)} = 0,
 $$
 and by applying Cauchy-Schwarz inequality and \eqref{Young},
 $$
   \a{\beta\ip{\ol\zeta}{\ol\eta}} \leq \half[\beta]\nH{\eta}^2 + \half[\beta]\nH{\zeta}^2.
 $$
 Rewriting 
 $\ip{\oI(\eta_1)}{-\ol\eta_1} = \ip{\oI(\eta_1)-\eta_1}{-\ol\eta_1} + \ip{\eta_1}{\ol\eta_1},$
 we have,
 $$ - \mu\ip{\oI(\eta_1)}{-\ol \eta_1} = 
    -\mu\ip{\oI(\eta_1)-\eta_1}{-\ol\eta_1} - \mu\nh{\eta_1}^2,
 $$
 so we obtain:
 \begin{multline}
   \half\dt\nh\eta^2 + \left(\alpha-\half[\beta]\right)\nH\eta^2 
   	- \half[\beta]\nH\zeta^2 + \mu\nh{\eta_1}^2 \\
   	\leq \a{\ib{\zeta}{v}{\ol\eta} } + \a{\ib{\wt}{\eta}{\ol\eta}}
    + \mu\a{\ip{\oI(\eta_1)-\eta_1}{\ol\eta_1}}.
 \end{multline}
 By the properties of $\oI$, we have
 \begin{align*}
	\mu\a{\ip{\oI(\eta_1)-\eta_1}{\ol\eta_1}} &\leq \mu\nl{\oI(\eta_1)-\eta_1}\nH{\eta_1}\\
    &\leq \mu\left(\cIa h\nh{\eta_1}+\cIb h^2\nH{\eta_1}\right)\nH{\eta_1}\\
    &\leq \frac{\mu^2}{2(\alpha-\beta)}(\cIa h\nh{\eta_1} 
		+ \cIb h^2\nH{\eta_1})^2 + \half[\alpha-\beta]\nH{\eta_1}^2\\
    &\leq \frac{\mu^2\cIa^2 h^2}{\alpha-\beta}\nh{\eta_1}^2 
		+ \frac{\mu^2\cIb^2 h^4}{\alpha-\beta}\nH{\eta_1}^2
		+ \half[\alpha-\beta]\nH{\eta_1}^2.
 \end{align*}
 Therefore,
 \begin{multline}\label{T2 before non lin}
 \half\dt\nh\eta^2 
 	+ \left(\half[\alpha]-\frac{\mu^2\cIb^2 h^4}
 		{\alpha-\beta}\right)\nH\eta^2 
	- \half[\beta]\nH\zeta^2 + \mu\left(1- 
 		\frac{\mu\cIa^2 h^2}{\alpha-\beta}\right)\nh{\eta_1}^2 \\
	\leq \a{\ib{\zeta}{v}{\ol\eta}} + \a{\ib{\wt}{\eta}{\ol\eta}}.
 \end{multline}
 
 Note that 
 $1- \frac{\mu\cIa^2 h^2}{\alpha-\beta} > \half,$
 and 
 $\frac{\mu^2\cIb^2 h^4}{\alpha-\beta} < \frac{\alpha-\beta}{4}$
 as long as 
 \begin{equation}\label{h mu restriction}
	h^2 < \frac{\alpha-\beta}{2\mu\max\{\cIa^2,\cIb\}}.
 \end{equation}
 
 Now we estimate the nonlinear terms using 
 Lemma \ref{lemma nonlinear term est 2}.
 By \eqref{nonlin est 1 T2}, we obtain
 \begin{multline*}
 	\a{ \ib \zeta v {\ol \eta} } \leq
 		3\cT\nh{\zeta_1}\nh{v}\nH{\eta}\lnTwo{\zeta}\\
    	+(\cT+4\cB)\nH{\zeta}\nh{v}\nh{\eta_1}\lnTwo{\eta}\\
   	+ 2\cT\nh{\zeta}\nH{v}\nh{\eta_1}\lnTwo{\eta},
 \end{multline*}
 so by applying \eqref{Young}, we obtain
 \begin{multline*}
\a{ \ib \zeta v {\ol \eta} } \leq 
	\frac{\alpha-\beta}{32}\left(\nH{\eta}^2+\nH{\zeta}^2
		+4\pi^2\nh{\zeta}^2\right)\\
  	+\frac{72\cT^2}{(\alpha-\beta)}\nh{\zeta_1}^2\nh{v}^2\lntwo{\zeta}\\
           +\frac{64(1+4\pi^2)(\cT^2+\cB^2)}{4\pi^2(\alpha-\beta)}
    	\left(\nh{v}^2+\nH{v}^2\right)\nh{\eta_1}^2\lntwo{\eta}.
 \end{multline*}
 Also, we use \eqref{Poincare} to write 
 $4\pi^2\nh{\zeta}^2\leq \nH{\zeta}$.
 
 For the other term, we first apply \eqref{nonlin est 2 T2}, and obtain
 $$
 \a{\ib{\wt}{\eta}{\ol\eta}} \leq 
 	(2\cB+5\cT)\nh{\wt}\nh{\eta_1}\nH{\eta}\lnTwo{\eta}.$$
 Then, by \eqref{Young}, we have
 $$
 \a{\ib{\wt}{\eta}{\ol\eta}} \leq \frac{\alpha-\beta}{32}\nH{\eta}^2
	 +\frac{200(\cB+\cT)^2}{\alpha-\beta}\nh{\wt}^2\nh{\eta_1}^2\lntwo{\eta}.
 $$
 
 \newcommand{\gC}{\gamma_0}
 Combining these estimates with \eqref{T2 before non lin}, we have:
 \begin{multline}\label{T2 after non lin}
 \half\dt\nh\eta^2 
 	+ \left(\half[\alpha]-\frac{5(\alpha-\beta)}{16}\right)\nH\eta^2 
	- \left(\half[\beta]+\frac{\alpha-\beta}{16}\right)\nH\zeta^2 \\
	+ \left[\half[\mu]-\gC\left(\nh{\wt}^2+\nh{v}^2+\nH{v}^2\right)
		\lntwo{\eta}\right]\nh{\eta_1}^2\\
	- \gC\nh{v}^2\lntwo{\zeta}\nh{\zeta_1}^2\leq 0,
 \end{multline}
 where 
 $$\gC:=\frac{200(\cB+\cT)^2}{\nub}=\max\left\{\frac{72\cT^2}{(\alpha-\beta)},
 \frac{64(1+4\pi^2)(\cT^2+\cB^2)}{4\pi^2(\alpha-\beta)},
 \frac{200(\cB+\cT)^2}{\alpha-\beta}\right\}.$$
 
 Adding \eqref{T2 after non lin} with the corresponding inequality for 
 $\dt\nh\zeta^2$, we obtain:
 \begin{multline}\label{T2 before log}
 \half\dt\nh\eta^2 + \half\dt\nh\zeta^2
 	+ \frac{\alpha-\beta}{8}\left(\nH\eta^2+\nH\zeta^2\right)\\
 	+ \left[\half[\mu]-\gC
 		\left(\nh{\wt}^2+\nh{v}^2+\nh{w}^2+\nH{v}^2\right)
 		\lntwo{\eta}\right]\nh{\eta_1}^2\\
 	+ \left[\half[\mu]-\gC
 		\left(\nh{\vt}^2+\nh{v}^2+\nh{w}^2+\nH{w}^2\right)
 		\lntwo{\zeta}\right]\nh{\zeta_1}^2\leq 0.
 \end{multline}
 Next, we write
 $$
 \frac{\alpha-\beta}{8}\nH\eta^2 \geq \frac{\alpha-\beta}{16}\nH\eta^2
 + \frac{\alpha-\beta}{16}\arsT{\eta}4\pi^2\nh{\eta_1}^2
 $$
 and
 $$
 \frac{\alpha-\beta}{8}\nH\zeta^2 \geq \frac{\alpha-\beta}{16}\nH\zeta^2
 + \frac{\alpha-\beta}{16}\arsT{\zeta}4\pi^2\nh{\zeta_1}^2.
 $$
 Then, by defining $$r(u)=\arsT{u}$$ and 
 $$\gamma=\frac{4}{\pi^2\nua}\gC\left(\nh{\vt}^2+\nh{\wt}^2+\nh{v}^2+\nh{w}^2
	 +\nH{v}^2+\nH{w}^2\right),$$
 by \eqref{Poincare} we can rewrite \eqref{T2 before log} as:
 \begin{multline*}
 \half\dt(\nh\eta^2+\nh\zeta^2)
 	+ \frac{\alpha-\beta}{16}\left(\nH\eta^2+\nH\zeta^2\right)\\
 	+ \left[\half[\mu]+
	  \frac{\pi^2\nua}{4}(r(\eta)-\gamma\ag[1+]{r(\eta)})\right]
	  	\nh{\eta_1}^2\\
 	+ \left[\half[\mu]+
	  \frac{\pi^2\nua}{4}(r(\zeta)-\gamma \ag[1+]{r(\zeta)})\right]
	  	\nh{\zeta_1}^2 \leq 0.
 \end{multline*}
 Now we apply Lemma \ref{lemma log min} and conclude that
 \begin{multline*}
 \half\dt(\nh\eta^2+\nh\zeta^2)
  	+ \frac{\alpha-\beta}{16}\left(\nH\eta^2+\nH\zeta^2\right)\\
  	+ \left[\half[\mu] -\frac{\pi^2\nua}{4}\gamma\ln(\gamma)\right]\nh{\eta_1}^2
  	+ \left[\half[\mu] -\frac{\pi^2\nua}{4}\gamma\ln(\gamma)\right]\nh{\zeta_1}^2\leq0.
 \end{multline*}
 Using \eqref{Poincare} again, we have 
 $$\nH\eta^2+\nH\zeta^2 \geq 4\pi^2(\nh\eta^2+\nh\zeta^2),$$
 so by defining 
 $$Y=\nh{\eta}^2+\nh{\zeta}^2,$$ 
 and 
 $$ \psi=\min\left\{\frac{\pi^2(\alpha-\beta)}{2},\:
 	\mu - \frac{\pi^2(\alpha-\beta)}{2}\gamma\ln(\gamma)\right\} $$
 we obtain:
 \begin{equation}
  \dt Y + \psi Y \leq 0.
 \end{equation}
 Thus, as long as we choose 
 $\mu > \frac{\pi^2(\alpha-\beta)}{2}(1+\gamma\ln(\gamma))$,
 we conclude by Gronwall's inequality that
 $$ 
 Y(t) \leq Y(0)e^{-\pi^2(\alpha-\beta)t/2},
 \quad \forall t\in[0,T_0).
 $$
 
 By \eqref{fin exist time sup}, \eqref{sup bound}, and \eqref{sup lap bound},
 $$ 
 \gamma \leq \frac{4}{\pi^2(\alpha-\beta)}\gamma_0\left( \cupb[60] G^2 e^{\C G^4} +
 \cupT G^2\left[1+\left(1+G^2e^{\C G^4}\right)
 	 \left(1+e^{\C G^4}+G^4e^{\C G^4}\right)\right] \right)<\infty,
 $$
 so on the time interval $[0,T_0)$, such a $\mu$ is available.
 Specifically, it is sufficient to choose
 \begin{equation}
 \mu \geq 2000(\cB+\cT)^2(20\pi^2+\cM)G^2(1+G^2)^3 e^{2\C G^4}
 \left(\tilde{c} + \ln(1+G) + \C G^4\right),
 \end{equation}
 where $\tilde{c}:= \ln(250(\cB+\cT)^2(20\pi^2+\cM))/8$, so
 \begin{equation}
 \mu\sim G^{12}e^{2\C G^4}.\label{T2 mu bound}
 \end{equation}
 
 Therefore, for all $t\in[0,T_0)$, we obtain
 $$ Y(t)\leq Y(0)\leq 
 	2\nh{v_0}^2 + 2\nh{\vt_0}^2 + 2\nh{w_0}^2 + 2\nh{\wt_0}^2
 	\leq \cupb[20] G^2 e^{\C G^4}.
 $$
 This implies that, in fact, 
 $$
 \sup_{t\in[0,T_0)}(\nh{\vt(t)}^2+\nh{\wt(t)}^2)\leq \cupb[40] G^2 e^{\C G^4},
 $$
 which is a contradiction to \eqref{fin exist time sup}.
 
 Hence we have $T_0=\infty,$ and $(\vt(t),\wt(t))$ converges exponentially
 in time to $(v(t),w(t))$ in the $\sH$ norm, and we have established the estimate:
  $$
  \sup_{t\in[0,\infty)}(\nh{\vt(t)}^2+\nh{\wt(t)}^2) \leq \cupb[50] G^2 e^{\C G^4}.
  $$
 
 Also, our restriction on $\mu$ \eqref{T2 mu bound} is in fact sufficient to guarantee convergence on $[0,\infty)$, with our restriction \eqref{h mu restriction} on $h$, which we see now means we can choose
 $$ h\sim G^{-6}e^{-\C G^4}. $$
 
 \end{proof}

\subsection{Determining Interpolants}

 \begin{proof}[{\bf Proof of Theorem~\ref{thm gen. DA}}]~\\
 The proof proceeds exactly as that of Theorem~\ref{thm all},
 where $\delta^{(1)}\equiv\delta^{(2)}
 \equiv\epsilon^{(1)}\equiv\epsilon^{(2)}\equiv 0$, 
 with a few differences.
 As before, we let $\eta=v-\vt$ and then we obtain a differential inequality for
 $\nl{\eta}$. We get the same inequality as before but with two extra terms.
 
 After subtracting the equations for $v$ and $\vt$, we have
 $f-(f+\delta^{(1)})=-\delta^{(1)}$ for the forcing term, 
 and after taking the inner product with $\eta$ we have 
	$$ 
	 \a{\Ip{\delta^{(1)}}{\eta}} \leq \nl{\delta^{(1)}}\nl{\eta} 
	 \leq \frac{1}{\mu}\nl{\delta^{(1)}}^2+\frac{\mu}{4} \nl{\eta}^2. 
	$$
 Also, we have $\mu\oI\left(u+\epsilon^{(1)}-\ut\right) = \mu\oI\left(u-\ut\right)
 	+\mu\oI\left(\epsilon^{(1)}\right)$,
 and after taking the inner product with $\eta$, we obtain 
 $$
	 \a{\mu\Ip {\oI(\epsilon^{(1)})} \eta} \leq \mu\nl{\oI(\epsilon^{(1)})}\nl \eta
	 \leq \mu\nl{\oI(\epsilon^{(1)})}^2 + \frac{\mu}{4} \nl{\eta}^2. 
 $$
 We have similar additions for the inequality we derive for 
 $\zeta:=w-\wt$.
 
 Thus, letting $Y(t)=\nl{\eta(t)}^2+\nl{\zeta(t)}^2$ and proceeding 
 as before, we eventually get:
 $$ \dt Y + \psi Y \leq \phi, $$
 where 
	 $$ \psi(t):= \frac{\mu}{2} - 
	 \left(\frac{\cL^4+(\alpha-\beta)^4}{2(\alpha-\beta)^3}\right)
	 \left(\nh v^2 + \nh w^2\right),
	 $$
	 and
	 $$ \phi(t):= \frac{1}{\mu}\left(\nl{\delta^{(1)}}^2 
	     +\nl{\delta^{(2)}}^2\right) 
	    +\mu\left(\nl{\oI(\epsilon^{(1)})}^2
	     +\nl{\oI(\epsilon^{(2)})}^2\right).$$
	 Since $\nl{\delta{(1)}},\nl{\delta{(2)}}\to0$ and
	 $\nl{\oI(\epsilon^{(1)})},\nl{\oI(\epsilon^{(2)})}\to 0$,
	 we have $\nl{\phi}\to0$. 
 Therefore, by Proposition \ref{gronwall}, $\nl{v-\vt},\nl{w-\wt}\to 0$
 as $t\to\infty$.
 \end{proof}
 
 
 \begin{proof}[{\bf Proof of Theorem~\ref{thm det interp}}]~\\
 Let $\mu=\frac{\nua}{\cI^2h^2}$. Then $h$, $\oI$, and $\mu$ satisfy 
 Theorem~\ref{thm all} with $(v^{(1)},w^{(1)})$ as the reference solution.\\
 Let $(\vt,\wt)$ be the corresponding solution.\\
 Then $\nl{v^{(1)}(t)-\vt(t)}\to 0$ and $\nl{w^{(1)}(t)-\wt(t)}\to 0$, 
 and for some $\PT$, $\vt$ and $\wt$ satisfy the following equations:
 \begin{align*}
 \pt \vt - \alpha \ol \vt + \beta \ol \wt + \ob \wt \vt + \og \PT &= 
 f^{(1)} + \mu\oI\left(v^{(1)}-\vt\right)&\\
  &= f^{(2)} + (f^{(1)}-f^{(2)}) +
  \mu\oI\left(v^{(2)}+(v^{(1)}-v^{(2)})-\vt\right),&\\
 \pt \wt - \alpha \ol \wt + \beta \ol \vt + \ob \vt \wt + \og \PT &= 
 g^{(1)} + \mu\oI\left(w^{(1)}-\wt\right)&\\
  &= g^{(2)} + (g^{(1)}-g^{(2)}) +
  \mu\oI\left(w^{(2)}+(w^{(1)}-w^{(2)})-\wt\right).&
 \end{align*}
 Therefore, setting $\delta^{(1)}:=f^{(1)}-f^{(2)}$ and
 $\delta^{(2)}:=g^{(1)}-g^{(2)}$, and $\epsilon^{(1)}:=v^{(1)}-v^{(2)}$
 and $\epsilon^{(2)}:=w^{(1)}-w^{(2)}$, 
 we see that $(\vt,\wt)$ 
 must be the unique solution guaranteed by Theorem \ref{thm gen. DA}, 
 with $(v^{(2)},w^{(2)})$ as the reference solution.\\
 Therefore $\nl{v^{(2)}(t)-\vt(t)}\to 0$ and $\nl{w^{(2)}(t)-\wt(t)}\to 0$.\\
 Thus, 
 \begin{align*}
  \nl{v^{(1)}(t)-v^{(2)}(t)} \leq 
  \nl{v^{(1)}(t)-\vt(t)} + \nl{\vt(t)-v^{(2)}(t)} \to 0,
 \end{align*}
 and
 \begin{align*}
  \nl{w^{(1)}(t)-w^{(2)}(t)} \leq 
  \nl{w^{(1)}(t)-\wt(t)} + \nl{\wt(t)-w^{(2)}(t)} \to 0.&\\
 \end{align*}
 \end{proof}

\section{Concluding Remarks}
 We have shown that, in the language of the reformulated 
 equations, solutions $(\vt,\wt)$ of the data assimilation
 equations will converge to the corresponding true values
 $(v,w)$ in $\sL$, even if measurements are only taken 
 for only one of $v$ and $w$. This equates to having to 
 take measurements on either $u+b$ or $u-b$. Could one
 prove that it is sufficient to collect data on just $u$ or
 just $b$ and still get convergence, similar to the 
 result for the reformulated variables?
 
 If one were to consider collecting data only on
 the magnetic field, $b$, then the problem is evident when 
 we take $b(t)\equiv\tilde{b}(t)\equiv g\equiv 0$ 
 for all $t\geq0$, because we then have $u$ and $\tilde{u}$
 satisfying the Navier-Stokes equations with 
 different initial conditions and no data assimilation.
 Hence, there is an asymmetry between the original system and the reformulated system.
 
 The answer to the question for collecting data on the velocity field, $u$, is open. However, since we've demonstrated that the algorithm works with knowledge of only the sum of measurements on $u$ and $b$, it may be that the knowledge of the velocity field is what makes this work, and so a $u$-measurement only algorithm is hopeful. However, since it seems we shouldn't be able to prove the convergence of a $b$-measurement only algorithm, and the Els\"asser variable formulation does not distinguish $u$ and $b$, a proof of a $u$-measurement only algorithm would have to be in terms of the original variables.

\section{Appendix}
\begin{proof}[{\bf Proof of Proposition~\ref{prop upper bounds} }]
We provide only a formal proof of \eqref{int bound} here.  A rigorous proof can be carried out by, e.g., first proving the bounds at the level of finite-dimensional Galerkin truncation, and then passing to a limit. 

Taking a (formal) inner-product of \eqref{MHD_v} with $v$, and of \eqref{MHD_w} with $w$, using \eqref{MHD_vw_div} and adding the results, we obtain
\begin{align*}
&\qquad
\frac{1}{2}\frac{d}{dt}\left(\nl{v}^2 + \nl{w}^2\right)
+ (\alpha-\beta)\left(\nh{v}^2 + \nh{w}^2\right)
\\&\leq \ip f v + \ip g w 
\leq 
\nl{f}\nl{v}+\nl{g}\nl{w}
\\&\leq  
\tfrac{1}{8\pi^2(\alpha-\beta)}\left(\nl{f}^2 + \nl{g}^2\right)
+\tfrac{(\alpha-\beta)}{2}4\pi^2\left(\nl{v}^2 + \nl{w}^2\right)
\\&\leq  
\tfrac{1}{8\pi^2(\alpha-\beta)}\left(\nl{f}^2 + \nl{g}^2\right)
+\tfrac{(\alpha-\beta)}{2}\left(\nh{v}^2 + \nh{w}^2\right),
\end{align*}
where we used the Poincar\'e inequality and Young's inequality.  Therefore, after collecting terms,
\begin{align}\label{L2_vw1}
\frac{d}{dt}\left(\nl{v}^2 + \nl{w}^2\right)
+ (\alpha-\beta)\left(\nh{v}^2 + \nh{w}^2\right)\leq  
\tfrac{1}{4\pi^2(\alpha-\beta)}\left(\nl{f}^2 + \nl{g}^2\right),
\end{align}
and by using the Poincar\'e inequality on the left hand side,
\begin{align}\label{L2_vw}
\frac{d}{dt}\left(\nl{v}^2 + \nl{w}^2\right)
+ 4\pi^2(\alpha-\beta)\left(\nl{v}^2 + \nl{w}^2\right)\leq  
\tfrac{1}{4\pi^2(\alpha-\beta)}\left(\nl{f}^2 + \nl{g}^2\right).
\end{align}
Then by Gr\"onwall's inequality,
\begin{align}\label{L2_ptwise_bound}
\nl{v(t)}^2 + \nl{w(t)}^2
\leq 
(\nl{v(0)}^2 + \nl{w(0)}^2)e^{-4\pi^2(\alpha-\beta) t}
+ 
\tfrac{1}{16\pi^4(\alpha-\beta)^2} \esssup_{s\in[0,t]}
\left(\nl{f(s)}^2 + \nl{g(s)}^2\right).
\end{align}

Let $t_*>0$ be large enough so that 
\begin{align}\label{L2_ess_sup}
\esssup_{t\geq t_*}\left(\nl{f(t)}^2 + \nl{g(t)}^2\right)
\leq 2 \limsup_{t\to\infty}\left(\nl{f(t)}^2 + \nl{g(t)}^2\right),
\end{align}
and choose $t_0>t_*$ so that
$$ (\nl{v(t_*)}^2 + \nl{w(t_*)}^2)e^{-4\pi^2(\alpha-\beta) (t_0-t_*)} \leq 
\tfrac{3}{8\pi^4(\alpha-\beta)^2} \limsup_{t\to\infty} \left(\nl{f(t)}^2 + \nl{g(t)}^2\right).$$
Then by using Gr\"onwall's inequality again on \eqref{L2_vw} with initial time $t_*$,
we see that for all $t\geq t_0$,
\begin{align}
\nl{v(t)}^2 + \nl{w(t)}^2
&\leq 
(\nl{v(t_*)}^2 + \nl{w(t_*)}^2)e^{-4\pi^2(\alpha-\beta) (t-t_*)} + 
\tfrac{1}{16\pi^4(\alpha-\beta)^2} \esssup_{s\in[t_*,t]}
\left(\nl{f(s)}^2 + \nl{g(s)}^2\right)\nonumber \\
&\leq \tfrac{1}{2\pi^4(\alpha-\beta)^2} \limsup_{s\to\infty}\left(\nl{f(s)}^2 + \nl{g(s)}^2\right).\label{L2_vw_bound}
\end{align}

Next, integrating \eqref{L2_vw1} on $[t,t+T]$, and using \eqref{L2_ess_sup}, 
\begin{align*}
&\quad
\nl{v(t+T)}^2 + \nl{w(t+T)}^2
+ (\alpha-\beta)\int_t^{t+T}\left(\nh{v(s)}^2 + \nh{w(s)}^2\right)\,ds
\\&\leq  
\nl{v(t)}^2 + \nl{w(t)}^2
+
\tfrac{T}{2\pi^2(\alpha-\beta)}\limsup_{s\to\infty}\left(\nl{f(s)}^2 + \nl{g(s)}^2\right).
\end{align*}
Thus, using \eqref{L2_vw_bound}, for $t\geq t_0$, 
\begin{align}
\int_t^{t+T}\left(\nh{v(s)}^2 + \nh{w(s)}^2\right)\,ds
\leq
( 1 + \pi^2(\alpha-\beta)T )(\alpha-\beta)
\limsup_{s\to\infty}\frac{\nl{f(s)}^2 + \nl{g(s)}^2}{2\pi^4(\alpha-\beta)^4},
\end{align}
which implies \eqref{int bound}.
\end{proof}

\begin{proof}[{\bf Proof of Lemma~\ref{lemma nonlinear term est}}]
To show \eqref{nonlin est 2}, we first apply \eqref{Holder} and \eqref{Young} 
then \eqref{Lady} and \eqref{Young}:
\begin{align*}
 \a{\Ib u v w}&\leq\I{\a u \a{\og v} \a w}\leq \nh{v} \nL{u}\nL{w}&\\
 &\leq \frac{\delta}{2}\nh{v} \nL{u}^2 + \frac{1}{2\delta}\nh{v} \nL{w}^2&\\
 &\leq \frac{\cL\delta}{2}\nh{v}\nl{u}\nh{u} 
	+ \frac{\cL}{2\delta}\nh{v}\nl{w}\nh{w}.&\\
 &\leq \frac{\cL\delta}{2}\left(\half\nh{v}^2\nl{u}^2 + \half\nh{u}^2\right)
	+ \half\frac{\cL^2}{4\epsilon\delta^2}\nh{v}^2\nl{w}^2 
	+ \half[\epsilon]\nh{w}^2.
\end{align*}
We obtain \eqref{nonlin est 2.1} by switching the roles of $u$ and $w$ after 
applying \eqref{Holder}.\\
The proof of \eqref{nonlin est 1} requires us to estimate the components of the 
product differently. First, write
$$
 \a{\Ib u v w} = \a{\I{\sum_{i,j=1}^2 u_i\pxi v_j w_j}} \leq \sum_{i,j=1}^2
	\a{\I{u_i\pxi v_j w_j}},
$$
and then we estimate the terms of the sum separately.\\
(Case: $i=1,j=1$) For this case we proceed similarly as in the proof of 
\eqref{nonlin est 2}, to obtain:
\begin{align*}
 \a{\I{u_1\pn1 v_1 w_1}} & \leq \nh{v_1}\nL{u_1}\nL{w_1} 
 \\&\leq \half[\cL]\nh{v_1}\nl{u_1}\nh{u_1} 
	+ \half[\cL]\nh{v_1}\nl{w_1}\nh{w_1} &\\ 
 &\leq \frac{\cL\delta}{4}\nh{u_1}^2 
	+ \frac{\cL}{4\delta}\nh{v_1}^2\nl{u_1}^2
     + \frac{\cL\delta}{4}\nh{w_1}^2 
     + \frac{\cL}{4\delta}\nh{v_1}^2\nl{w_1}^2.
\end{align*}
(Case: $i=1,j=2$) For this and the next case, we use \eqref{Brezis}:
\begin{align*}
 \a{\I{u_1\pn1 v_2 w_2}} & \leq \cB\nh{w_2}\nh{v_2}\nl{u_1} 
 \left(1 + \ln\left(\ar u\right)\right)^{1/2} &\\
 &\leq \frac{\cB\delta}{2}\nh{w_2}^2 
	+ \frac{\cB}{2\delta}\nh{v_2}^2\nl{u_1}^2 
		\left(1 + \ln\left(\ar u\right)\right)\\
 &\leq \frac{\cB\delta}{2}\nh{w_2}^2 + \frac{\cB}{2\delta}\nh{v}^2\nl{u_1}^2 
		\left(1 + \ln\left(\ar u\right)\right)
\end{align*}
(Case: $i=2,j=1$) Similarly, we obtain:
\begin{align*}
 \a{\I{u_2\pn2 v_1 w_1}}
 &\leq \frac{\cB\delta}{2}\nh{u_2}^2 + \frac{\cB}{2\delta}\nh{v}^2\nl{w_1}^2 
  \left(1 + \ln\left(\ar w\right)\right) &
\end{align*}
(Case: $i=2,j=2$) Now we use the divergence free conditions 
(i.e. $\pn1 u_1 = -\pn2 u_2$) and integrate by parts in order to obtain integrals 
in which the second components of $u$ and $w$ do not appear together:
$$
 \I{u_2\pn2 v_2 w_2} = -\I{\pn2u_2 v_2 w_2} - \I{u_2 v_2 \pn2w_2}
 = \I{\pn1u_1 v_2 w_2} + \I{u_2 v_2 \pn1w_1} 
$$
$$
 = -\I{u_1 \pn1v_2 w_2} - \I{u_1 v_2 \pn1w_2} 
	- \I{\pn1u_2 v_2 w_1} - \I{u_2 \pn1v_2 w_1}.
$$
Now, each of these terms can be estimated similarly to the cases where 
$i\neq j:$
\begin{align*}
 \a{\I{u_1 \pn1v_2 w_2}}
  &\leq \frac{\cB\delta}{2}\nh{w_2}^2 + \frac{\cB}{2\delta}\nh{v}^2\nl{u_1}^2 
  \left(1 + \ln\left(\ar u\right)\right) &\\
 \a{\I{u_1 v_2 \pn1w_2}}
  &\leq \frac{\cB\delta}{2}\nh{w_2}^2 + \frac{\cB}{2\delta}\nh{v}^2\nl{u_1}^2 
  \left(1 + \ln\left(\ar u\right)\right) &\\
 \a{\I{\pn1u_2 v_2 w_1}}
  &\leq \frac{\cB\delta}{2}\nh{u_2}^2 + \frac{\cB}{2\delta}\nh{v}^2\nl{w_1}^2 
  \left(1 + \ln\left(\ar w\right)\right) &\\
 \a{\I{u_2 \pn1v_2 w_1}}
  &\leq \frac{\cB\delta}{2}\nh{u_2}^2 + \frac{\cB}{2\delta}\nh{v}^2\nl{w_1}^2 
  \left(1 + \ln\left(\ar w\right)\right) &
\end{align*}
Taking the sum of these 7 inequalities obtained from the 4 cases, we have:
\begin{multline*}
 \a{\Ib u v w} \leq \frac{\cL\delta}{4}\nh{u_1}^2 
	+ \frac{3\cB\delta}{2}\nh{u_2}^2
	+\frac{\cL\delta}{4}\nh{w_1}^2 + \frac{3\cB\delta}{2}\nh{w_2}^2\\
	+\frac{\cL}{4\delta}\nh{v}^2\nl{u_1}^2 
		+ \frac{\cL}{4\delta}\nh{v}^2\nl{w_1}^2\\
	+\frac{3\cB}{2\delta}\nh{v}^2\nl{u_1}^2
		\left(1 + \ln\left(\ar u\right)\right)
	+\frac{3\cB}{2\delta}\nh{v}^2\nl{w_1}^2
		\left(1 + \ln\left(\ar w\right)\right).
\end{multline*}
Setting 
$c = \max\{\frac{\cL}{4},\frac{3\cB}{2}\}$ now yields \eqref{nonlin est 1}.
\end{proof}

\begin{proof}[{\bf Proof of Lemma~\ref{lemma nonlinear term est 2}}]
We start by writing 
 \begin{align*}
  \Ib u v {\ol w} = &\Ibb u 1 {v_1} {\ol w_1} + \Ibb u 2 {v_1} {\ol w_1}\\
  				  + &\Ibb u 1 {v_2} {\ol w_2} + \Ibb u 2 {v_2} {\ol w_2}.
 \end{align*}
 Now we'll estimate each term individually.\\
 By \eqref{Titi-Brezis} we have:
 \begin{align}\label{T2 lem est a1}
  \a{ \Ibb{u}{1}{v_1}{\ol w_1} } &\leq 
   \cT\nh{u_1}\nh{v_1}\nH{w_1}\lnTwo{u} \nonumber\\
    &\leq \cT\nh{u_1}\nh{v}\nH{w}\lnTwo{u},
 \end{align}
 and
 \begin{align}\label{T2 lem est a2}
  \a{ \Ibb {u} {1} {v_2} {\ol w_2} } &\leq 
   \cT\nh{u_1}\nh{v_2}\nH{w_2}\lnTwo{u} \nonumber\\
    &\leq \cT\nh{u_1}\nh{v}\nH{w}\lnTwo{u}.
 \end{align}
 Using integration by parts and the divergence free condition, 
 we have:
 \begin{align*}
  \Ibb {u} {2} {v_1} {\ol w_1} &= -\I{\px u_2 \py v_1 \px w_1} 
   - \I{\py u_2 \py v_1 \py w_1}\\
   &\quad+\I{u_2 \pn{yy} v_2 \px w_1} - \I{u_2 \pn{yy} v_1 \py w_1},
 \end{align*}
 so applying \eqref{Brezis} to the first two integrals and
 \eqref{Titi-Brezis} to the second two, we obtain:
 \begin{multline}\label{T2 lem est a3}
  \a{ \Ibb{u}{2}{v_1}{\ol w_1} }\leq
   \cB\nH{u}\nh{v}\nh{w_1}\lnTwo{w}\\
  +\cT\nh{u}\nh{w_1}\nH{v}\lnTwo{w}.
 \end{multline}
 Again by integrating by parts and using the divergence free condition,
 we obtain
 \begin{align*} 
 	\Ibb {u} {2} {v_2} {\ol w_2} &= \I{\px u_1 v_2 {\ol w_2}}\\
 	 &\quad+ \I{{\ol u_2} v_2 \px w_1} + \I{u_2 {\ol v_2} \px w_1}\\
 	 &\quad+\: 2\I{\px u_2 \px v_2 \px w_1}\: +\: 2\I{\py u_2 \py v_2 \px w_1}.
 \end{align*}
 Now, estimating with \eqref{Brezis} and \eqref{Titi-Brezis} we have:
 \begin{multline}\label{T2 lem est a4}
   \a{ \Ibb {u} {2} {v_2} {\ol w_2} }\leq
    \cT\nh{u_1}\nh{v}\nH{w}\lnTwo{u}\\
   +\cT\nH{u}\nh{v}\nh{w_1}\lnTwo{w}\\
   +\cT\nh{u}\nH{v}\nh{w_1}\lnTwo{w}\\
   +4\cB\nH{u}\nh{v}\nh{w_1}\lnTwo{w}.
  \end{multline}
 Combining \eqref{T2 lem est a1}, \eqref{T2 lem est a2}, \eqref{T2 lem est a3}, 
 and \eqref{T2 lem est a4}, we obtain:
  \begin{multline*}
 	\a{ \Ib u v {\ol w} } \leq 3\cT\nh{u_1}\nh{v}\nH{w}\lnTwo{u}\\
 		+(\cT+4\cB)\nH{u}\nh{v}\nh{w_1}\lnTwo{w}\\
	 	+ 2\cT\nh{u}\nH{v}\nh{w_1}\lnTwo{w},
  \end{multline*}
 so (a) is proven.
        
 In order to prove (b), we first write 
 \begin{align*}
   \Ib u v {\ol v} = &\Ibb u 1 {v_1} {\ol v_1} + \Ibb u 2 {v_1} {\ol v_1}\\
 				   + &\Ibb u 1 {v_2} {\ol v_2} + \Ibb u 2 {v_2} {\ol v_2}.
 \end{align*}
 Similar to the proof of (a), we proceed to estimate each term individually
 by appealing to \eqref{Brezis} or \eqref{Titi-Brezis}, by integrating by
 parts and using the divergence free conditions.\\
 By applying \eqref{Titi-Brezis}, we have:
 \begin{align}\label{T2 lem est b1}
 	\a{ \Ibb u 1 {v_1} {\ol v_1} } \leq \cT\nh{u}\nh{v_1}\nH{v}\lnTwo{v},
 \end{align}
 and
 \begin{align}\label{T2 lem est b2}
	\a{ \Ibb u 2 {v_1} {\ol v_1} } \leq \cT\nh{u}\nh{v_1}\nH{v}\lnTwo{v},
 \end{align}
 and using the divergence free condition, we obtain
 \begin{align}\label{T2 lem est b4}
	\a{ \Ibb u 2 {v_2} {\ol v_2} } &= \a{ -\I{ u_2\px{v_1}{\ol v_2} } }\nonumber\\
	 &\leq \cT\nh{u}\nh{v_1}\nH{v}\lnTwo{v}.
 \end{align}
 To estimate the remaining integral, we write:
 \begin{align*}
	\Ibb u 1 {v_2} {\ol v_2} &= \Ibb u 1 {v_2} {\pn{xx} v_2} + \Ibb u 1 {v_2} {\pn{yy} v_2}.
 \end{align*}
 Now,
 \begin{align*}
 	\Ibb u 1 {v_2} {\pn{yy} v_2} &= -\Ibb u 1 {v_2} {\py\px v_1}\\
 	&= \I{ \px u_1 \px v_2 \py v_1 } + \I{ u_1 \pn{xx} v_2 \py v_1 },
 \end{align*}
 so 
 \begin{subequations}
 \begin{equation}\label{T2 lem est b31}
 	\a{ \Ibb u 1 {v_2} {\pn{yy} v_2} } \leq (\cB+\cT)\nh{u}\nH{v}\nh{v_1}\lnTwo{v}.
 \end{equation}
 For the other term, we have
 \begin{equation*}
 	\Ibb u 1 {v_2} {\pn{xx} v_2} = -\I{ \px u_1 \px v_2 \px v_2 } -\I{ u_1 \pn{xx}v_2 \px v_2 },
 \end{equation*}
 so,
 \begin{equation*}
 	\Ibb u 1 {v_2} {\pn{xx} v_2} = -\half\I{ \px u_1 \px v_2 \px v_2 }.
 \end{equation*}
 Next, 
 \begin{align*}
 	-\half\I{ \px u_1 \px v_2 \px v_2 } = \half\I{ \py u_2 \px v_2 \px v_2 }\\
 	= -\I{ u_2 \px\py v_2 \px v_2 } = \I{ u_2 \pn{xx}v_1 \px v_2 }\\ 
 	= -\I{ \px u_2 \px v_1 \px v_2 } -\I{ u_2 \px v_1 \pn{xx} v_2 }.
 \end{align*}
 Therefore,
 \begin{equation}\label{T2 lem est b32}
  	\a{ \Ibb u 1 {v_2} {\pn{xx} v_2} } 
  	\leq (\cB+\cT)\nh{u}\nh{v_1}\nH{v}\lnTwo{v}.
 \end{equation}
 \end{subequations} 
 
 Hence, by combining \eqref{T2 lem est b1}, \eqref{T2 lem est b2}, \eqref{T2 lem est b4}, 
 \eqref{T2 lem est b31}, and \eqref{T2 lem est b32}, we obtain:
 \begin{equation*}
	\a{ \Ib u v {\ol v} } \leq (2\cB+5\cT)\nh{u}\nh{v_1}\nH{v}\lnTwo{v},
 \end{equation*}
 as claimed.

\end{proof}

\bibliographystyle{alpha}
\bibliography{References}

\newcommand{\etalchar}[1]{$^{#1}$}
\begin{thebibliography}{GHKVZ14}

\bibitem[ANLT16]{Albanez_Nussenzveig_Lopes_Titi_2016}
D{\'e}bora~AF Albanez, Helena~J Nussenzveig~Lopes, and Edriss~S Titi.
\newblock Continuous data assimilation for the three-dimensional
  {N}avier--{S}tokes-$\alpha$ model.
\newblock {\em Asymptotic Anal.}, 97(1-2):139--164, 2016.

\bibitem[AOT14]{Azouani_Olson_Titi_2014}
Abderrahim Azouani, Eric Olson, and Edriss~S. Titi.
\newblock Continuous data assimilation using general interpolant observables.
\newblock {\em J. Nonlinear Sci.}, 24(2):277--304, 2014.

\bibitem[ASZL15]{Adams_Stone_Zimmerman_Lathrop_2015}
Matthew~M. Adams, Douglas~R. Stone, Daniel~S. Zimmerman, and Daniel~P. Lathrop.
\newblock Liquid sodium models of the earth's core.
\newblock {\em Progress in Earth and Planetary Science}, 2(1):1--18, 2015.

\bibitem[AT14]{Azouani_Titi_2014}
Abderrahim Azouani and Edriss~S. Titi.
\newblock Feedback control of nonlinear dissipative systems by finite
  determining parameters---a reaction-diffusion paradigm.
\newblock {\em Evol. Equ. Control Theory}, 3(4):579--594, 2014.

\bibitem[ATK{\etalchar{+}}15]{Altaf_Titi_Knio_Zhao_Mc_Cabe_Hoteit_2015}
M.~U. {Altaf}, E.~S. {Titi}, O.~M. {Knio}, L.~{Zhao}, M.~F. {McCabe}, and
  I.~{Hoteit}.
\newblock Downscaling the 2{D} {B}enard convection equations using continuous
  data assimilation.
\newblock {\em ArXiv e-prints}, December 2015.

\bibitem[BGJ89]{Bloxham_Gubbins_Jackson_1989}
Jeremy Bloxham, David Gubbins, and Andrew Jackson.
\newblock Geomagnetic secular variation.
\newblock {\em Phil. Trans. R. Soc. A}, 329(1606):415--502, 1989.

\bibitem[BM17]{Biswas_Martinez_2017}
Animikh Biswas and Vincent~R. Martinez.
\newblock Higher-order synchronization for a data assimilation algorithm for
  the 2{D} {N}avier--{S}tokes equations.
\newblock {\em Nonlinear Anal. Real World Appl.}, 35:132--157, 2017.

\bibitem[BOT15]{Bessaih_Olson_Titi_2015}
Hakima Bessaih, Eric Olson, and Edriss~S. Titi.
\newblock Continuous data assimilation with stochastically noisy data.
\newblock {\em Nonlinearity}, 28(3):729--753, 2015.

\bibitem[BRB02]{Bunge_Richards_Baumgardner_2002}
Hans-Peter Bunge, M.~A. Richards, and J.~R. Baumgardner.
\newblock Mantle-circulation models with sequential data assimilation:
  inferring present-day mantle structure from plate-motion histories.
\newblock {\em Phil. Trans. R. Soc. A}, 360(1800):2545--2567, 2002.

\bibitem[CF88]{Constantin_Foias_1988}
P.~Constantin and C.~Foias.
\newblock {\em Navier-{S}tokes {E}quations}.
\newblock Chicago Lectures in Mathematics. University of Chicago Press,
  Chicago, IL, 1988.

\bibitem[Cha61]{Chandrasekhar_1961}
S.~Chandrasekhar.
\newblock {\em Hydrodynamic and {H}ydromagnetic {S}tability}.
\newblock The International Series of Monographs on Physics. Clarendon Press,
  Oxford, 1961.

\bibitem[CJL00]{Constable_Johnson_Lund_2000}
Catherine~G Constable, Catherine~L Johnson, and Steven~P Lund.
\newblock Global geomagnetic field models for the past 3000 years: transient or
  permanent flux lobes?
\newblock {\em Phil. Trans. R. Soc. A}, 358(1768):991--1008, 2000.

\bibitem[CJT95]{cjt}
Bernardo Cockburn, Don~A. Jones, and Edriss~S. Titi.
\newblock Determining degrees of freedom for nonlinear dissipative equations.
\newblock {\em C. R. Acad. Sci. Paris S\'er. I Math.}, 321(5):563--568, 1995.

\bibitem[CKT01]{Cao_Kevrekidis_Titi_2001}
Chongsheng Cao, Ioannis~G. Kevrekidis, and Edriss~S. Titi.
\newblock Numerical criterion for the stabilization of steady states of the
  {N}avier--{S}tokes equations.
\newblock {\em Indiana Univ. Math. J.}, 50(Special Issue):37--96, 2001.
\newblock Dedicated to Professors Ciprian Foias and Roger Temam (Bloomington,
  IN, 2000).

\bibitem[CRB07]{Chandrasekar_Ridley_Bernstein_2007}
J.~Chandrasekar, A.J. Ridley, and D.S. Bernstein.
\newblock A comparison of the extended and unscented {K}alman filters for
  discrete-time systems with nondifferentiable dynamics.
\newblock {\em Proc. of the 2007 American Control Conference}, 2:4431--4436,
  2007.

\bibitem[Dal93]{Daley_1993_atmospheric_book}
R.~Daley.
\newblock {\em Atmospheric Data Analysis}.
\newblock Cambridge Atmospheric and Space Science Series. Cambridge University
  Press, 1993.

\bibitem[Dav01]{Davidson_2001}
P.~A. Davidson.
\newblock {\em An {I}ntroduction to {M}agnetohydrodynamics}.
\newblock Cambridge Texts in Applied Mathematics. Cambridge University Press,
  Cambridge, 2001.

\bibitem[DL72]{Duvaut_Lions_1972}
G.~Duvaut and J.-L. Lions.
\newblock In\'equations en thermo\'elasticit\'e et magn\'eto\-hydrodynamique.
\newblock {\em Arch. Rational Mech. Anal.}, 46:241--279, 1972.

\bibitem[Els50]{Elsasser_1950}
Walter~M. Elsasser.
\newblock The hydromagnetic equations.
\newblock {\em Phys. Rev.}, 79:183, Jul 1950.

\bibitem[FDKT14]{Foyash_Dzholli_Kravchenko_Titi_2014}
K.~Foyash, M.~S. Dzholli, R.~Kravchenko, and {\`E}.~S. Titi.
\newblock A unified approach to the construction of defining forms for a
  two-dimensional system of {N}avier-{S}tokes equations: the case of general
  interpolating operators.
\newblock {\em Uspekhi Mat. Nauk}, 69(2(416)):177--200, 2014.

\bibitem[FEA07]{Fournier_Eymin_Thierry_2007}
A.~{Fournier}, C.~{Eymin}, and T.~{Alboussi{\`e}re}.
\newblock A case for variational geomagnetic data assimilation: insights from a
  one-dimensional, nonlinear, and sparsely observed mhd system.
\newblock {\em Nonlinear Proc. Geoph.}, 14:163--180, April 2007.

\bibitem[FHJ{\etalchar{+}}10]{Fournier_Hulot_Jault_Kuang_Tangborn_Gillet_Canet_Aubert_Lhuillier_2010}
Alexandre Fournier, Gauthier Hulot, Dominique Jault, Weijia Kuang, Andrew
  Tangborn, Nicolas Gillet, Elisabeth Canet, Julien Aubert, and Florian
  Lhuillier.
\newblock An introduction to data assimilation and predictability in
  geomagnetism.
\newblock {\em Space Sci. Rev.}, 155(1):247--291, 2010.

\bibitem[FJT15]{Farhat_Jolly_Titi_2015}
Aseel Farhat, Michael~S. Jolly, and Edriss~S. Titi.
\newblock Continuous data assimilation for the 2{D} {B}\'enard convection
  through velocity measurements alone.
\newblock {\em Phys. D}, 303:59--66, 2015.

\bibitem[FLT16a]{Farhat_Lunasin_Titi_2016abridged}
Aseel Farhat, Evelyn Lunasin, and Edriss~S Titi.
\newblock Abridged continuous data assimilation for the 2{D} {N}avier--{S}tokes
  equations utilizing measurements of only one component of the velocity field.
\newblock {\em J. Math. Fluid Mech.}, 18(1):1--23, 2016.

\bibitem[FLT16b]{Farhat_Lunasin_Titi_2016benard}
Aseel Farhat, Evelyn Lunasin, and Edriss~S Titi.
\newblock Data assimilation algorithm for 3{D} {B\'e}nard convection in porous
  media employing only temperature measurements.
\newblock {\em Journal of Mathematical Analysis and Applications},
  438(1):492--506, 2016.

\bibitem[FLT16c]{Farhat_Lunasin_Titi_2016_Charney}
Aseel Farhat, Evelyn Lunasin, and Edriss~S Titi.
\newblock On the {C}harney conjecture of data assimilation employing
  temperature measurements alone: The paradigm of 3d planetary geostrophic
  model.
\newblock {\em arXiv:1608.04770}, 2016.

\bibitem[FLT17]{Farhat_Lunasin_Titi_2017_Horizontal}
Aseel Farhat, Evelyn Lunasin, and Edriss~S. Titi.
\newblock Continuous data assimilation for a {2D} {B}\'enard convection system
  through horizontal velocity measurements alone.
\newblock {\em Journal of Nonlinear Science}, pages 1--23, 2017.

\bibitem[FMRT01]{Foias_Manley_Rosa_Temam_2001}
C.~Foias, O.~Manley, R.~Rosa, and R.~Temam.
\newblock {\em Navier-{S}tokes {E}quations and {T}urbulence}, volume~83 of {\em
  Encyclopedia of Mathematics and its Applications}.
\newblock Cambridge University Press, Cambridge, 2001.

\bibitem[FMT16]{Foias_Mondaini_Titi_2016}
Ciprian Foias, Cecilia~F. Mondaini, and Edriss~S. Titi.
\newblock A discrete data assimilation scheme for the solutions of the
  two-dimensional {N}avier-{S}tokes equations and their statistics.
\newblock {\em SIAM J. Appl. Dyn. Syst.}, 15(4):2109--2142, 2016.

\bibitem[FMTT83]{Foias_Manley_Temam_Treve_1983}
C.~Foias, O.~Manley, R.~Temam, and Y.M. Tr\`eve.
\newblock Asymptotic analysis of the navier-stokes equations.
\newblock {\em Physica D: Nonlinear Phenomena}, 9(1-2):157--188, 1983.

\bibitem[FsP67]{foias-prodi}
C.~Foia\c~s and G.~Prodi.
\newblock Sur le comportement global des solutions non-stationnaires des
  \'equations de {N}avier-{S}tokes en dimension {$2$}.
\newblock {\em Rend. Sem. Mat. Univ. Padova}, 39:1--34, 1967.

\bibitem[FT84]{ftdet}
Ciprian Foias and Roger Temam.
\newblock Determination of the solutions of the {N}avier-{S}tokes equations by
  a set of nodal values.
\newblock {\em Math. Comp.}, 43(167):117--133, 1984.

\bibitem[FT91]{Foias_Titi_1991_Nonlinearity}
Ciprian Foias and Edriss~S. Titi.
\newblock Determining nodes, finite difference schemes and inertial manifolds.
\newblock {\em Nonlinearity}, 4(1):135--153, 1991.

\bibitem[GDZGP00]{Groth_DeZeeuw_Gombosi_Powell_2000}
Clinton Groth, Darren~L De~Zeeuw, Tamas~I Gombosi, and Kenneth~G Powell.
\newblock Global three-dimensional {MHD} simulation of a space weather event:
  {CME} formation, interplanetary propagation, and interaction with the
  magnetosphere.
\newblock {\em J. Geophys. Res.: Space Phys.}, 105(A11):25053--25078, 2000.

\bibitem[GHKVZ14]{GlattHoltz_Kukavica_Vicol_2014}
Nathan Glatt-Holtz, Igor Kukavica, Vlad Vicol, and Mohammed Ziane.
\newblock Existence and regularity of invariant measures for the three
  dimensional stochastic primitive equations.
\newblock {\em J. Math. Phys.}, 55(5):051504, 34, 2014.

\bibitem[GOT16]{Gesho_Olson_Titi_2015}
Masakazu Gesho, Eric Olson, and Edriss~S. Titi.
\newblock A computational study of a data assimilation algorithm for the
  two-dimensional {N}avier-{S}tokes equations.
\newblock {\em Commun. Comput. Phys.}, 19(4):1094--1110, 2016.

\bibitem[HOT11]{Hayden_Olson_Titi_2011}
Kevin Hayden, Eric Olson, and Edriss~S. Titi.
\newblock Discrete data assimilation in the {L}orenz and 2{D} {N}avier-{S}tokes
  equations.
\newblock {\em Phys. D}, 240(18):1416--1425, 2011.

\bibitem[JMT17]{Jolly_Martinez_Titi_2017}
Michael~S. Jolly, Vincent~R. Martinez, and Edriss~S. Titi.
\newblock A data assimilation algorithm for the subcritical surface
  quasi-geostrophic equation.
\newblock {\em Adv. Nonlinear Stud.}, 17(1):167--192, 2017.

\bibitem[JST15]{Jolly_Sadigov_Titi_2015}
Michael~S. Jolly, Tural Sadigov, and Edriss~S. Titi.
\newblock A determining form for the damped driven nonlinear {S}chr\"odinger
  equation---{F}ourier modes case.
\newblock {\em J. Differential Equations}, 258(8):2711--2744, 2015.

\bibitem[JT92a]{jt2}
Don~A. Jones and Edriss~S. Titi.
\newblock Determining finite volume elements for the {$2$}{D} {N}avier-{S}tokes
  equations.
\newblock {\em Phys. D}, 60(1-4):165--174, 1992.
\newblock Experimental mathematics: computational issues in nonlinear science
  (Los Alamos, NM, 1991).

\bibitem[JT92b]{jt3}
Don~A. Jones and Edriss~S. Titi.
\newblock On the number of determining nodes for the {$2$}{D} {N}avier-{S}tokes
  equations.
\newblock {\em J. Math. Anal. Appl.}, 168(1):72--88, 1992.

\bibitem[JT93]{jt1}
Don~A. Jones and Edriss~S. Titi.
\newblock Upper bounds on the number of determining modes, nodes, and volume
  elements for the {N}avier-{S}tokes equations.
\newblock {\em Indiana Univ. Math. J.}, 42(3):875--887, 1993.

\bibitem[Kal03]{Kalnay_2003_DA_book}
E.~Kalnay.
\newblock {\em Atmospheric Modeling, Data Assimilation and Predictability}.
\newblock Cambridge University Press, 2003.

\bibitem[LJL14]{Li_Jackson_Livermore_2014}
Kuan Li, Andrew Jackson, and Philip~W. Livermore.
\newblock Variational data assimilation for a forced, inertia-free
  magnetohydrodynamic dynamo model.
\newblock {\em Geophys. J. Int.}, 199(3):1662--1676, 2014.

\bibitem[LP17]{Larios_Pei_nonlinear_da_2017}
Adam Larios and Yuan Pei.
\newblock Nonlinear continuous data assimilation.
\newblock {\em arXiv:1703.03546}, 2017.

\bibitem[LSZ15]{Law_Stuart_Zygalakis_2015_book}
Kody Law, Andrew Stuart, and Konstantinos Zygalakis.
\newblock {\em A {M}athematical {I}ntroduction to Data {A}ssimilation},
  volume~62 of {\em Texts in Applied Mathematics}.
\newblock Springer, Cham, 2015.

\bibitem[LT]{Lunasin_Titi_2015}
E.~Lunasin and E.~S. Titi.
\newblock Finite determining parameters feedback control for distributed
  nonlinear dissipative systems - a computational study.
\newblock arXiv:1506.03709 [math.AP], (2015).

\bibitem[MDMB06]{Mendoza_DeMoor_Bernstein_2006}
O.~Barrero Mendoza, B.~De~Moor, and D.~S. Bernstein.
\newblock Data assimilation for magnetohydrodynamics systems.
\newblock {\em J. Comput. Appl. Math.}, 189(1-2):242--259, 2006.

\bibitem[MT16]{Mondaini_Titi_2017}
C.~F. {Mondaini} and E.~S. {Titi}.
\newblock Postprocessing {G}alerkin method applied to a data assimilation
  algorithm: a uniform in time error estimate.
\newblock 2016.
\newblock (arXiv 1612.06998).

\bibitem[MTT16]{Markowich_Titi_Trabelsi_2016}
Peter~A Markowich, Edriss~S Titi, and Saber Trabelsi.
\newblock Continuous data assimilation for the three-dimensional
  {B}rinkman--{F}orchheimer-extended {D}arcy model.
\newblock {\em Nonlinearity}, 29(4):1292, 2016.

\bibitem[OT03]{Olson_Titi_2003}
Eric Olson and Edriss~S. Titi.
\newblock Determining modes for continuous data assimilation in 2{D}
  turbulence.
\newblock {\em J. Statist. Phys.}, 113(5-6):799--840, 2003.
\newblock Progress in statistical hydrodynamics (Santa Fe, NM, 2002).

\bibitem[SML16]{Scott_McKenzie_Longcope_2016}
Roger~B. Scott, David~E. McKenzie, and Dana~W. Longcope.
\newblock Inferring the magnetohydrodynamic structure of solar flare
  supra-arcade plasmas from a data-assimilated field transport model.
\newblock {\em Astrophys. J.}, 819(1):56, 2016.

\bibitem[SOL02]{Sabaka_Olsen_Langel_2002}
Terence~J Sabaka, Nils Olsen, and Robert~A Langel.
\newblock A comprehensive model of the quiet-time, near-earth magnetic field:
  phase 3.
\newblock {\em Geophys. J. Int.}, 151(1):32--68, 2002.

\bibitem[ST83]{Sermange_Temam1983}
Michel Sermange and Roger Temam.
\newblock Some mathematical questions related to the {MHD} equations.
\newblock {\em Comm. Pure Appl. Math.}, 36(5):635--664, 1983.

\bibitem[STK07]{Sun_Tangborn_Kuang_2007}
Z~Sun, A~Tangborn, and W~Kuang.
\newblock Data assimilation in a sparsely observed one-dimensional modeled mhd
  system.
\newblock {\em Nonlinear Proc. Geoph.}, 14(2):181--192, 2007.

\bibitem[Tem01]{Temam_2001_Th_Num}
R.~Temam.
\newblock {\em Navier-{S}tokes {E}quations: {T}heory and {N}umerical
  {A}nalysis}.
\newblock AMS Chelsea Publishing, Providence, RI, 2001.
\newblock Theory and numerical analysis, Reprint of the 1984 edition.

\bibitem[Tit87]{Titi1987}
Edriss~S. Titi.
\newblock On a criterion for locating stable stationary solutions to the
  navier-stokes equations.
\newblock {\em Nonlinear Analysis: Theory, Methods {\&} Applications},
  11(9):1085--1102, jan 1987.

\bibitem[TRT{\etalchar{+}}08]{Teixeira_Ridley_Torres_Aguirre_Bernstein_2008}
Bruno~OS Teixeira, Aaron Ridley, Leonardo~AB T{\^o}rres, Luis~A Aguirre, and
  Dennis~S Bernstein.
\newblock Data assimilation for magnetohydrodynamics with a zero-divergence
  constraint on the magnetic field.
\newblock In {\em 2008 American Control Conference}, pages 2534--2539. IEEE,
  2008.

\end{thebibliography}
\end{document}